\documentclass[11pt,oneside,onecolumn,reqno]{amsart}

\usepackage{amssymb,latexsym,mathrsfs,
amsxtra,amscd,amsfonts,amsthm,amsmath,verbatim,epsfig,xypic}

\usepackage{color}
\usepackage[all]{xy}
\usepackage[colorlinks,pagebackref=true]{hyperref}

\usepackage[english]{babel}

\newcommand{\arxiv}[1]{\href{http://arxiv.org/abs/#1}{{\tt arXiv:#1}}}

\makeatletter

\makeatletter
\def\widebreve#1{\mathop{\vbox{\m@th\ialign{##\crcr\noalign{\kern3\p@}%
      \brevefill\crcr\noalign{\kern3\p@\nointerlineskip}%
      $\hfil\displaystyle{#1}\hfil$\crcr}}}\limits}

\def\brevefill{$\m@th \setbox\z@\hbox{$\braceld$}%
  \bracelu\leaders\vrule \@height\ht\z@ \@depth\z@\hfill\braceru$}

\makeatletter

\def\@citecolor{blue}
\def\@linkcolor{blue}
\def\@urlcolor{blue}
\def\@urlcolor{blue}

\numberwithin{equation}{section}
\def\im{\operatorname{im}}

\def\Proj{\operatorname{Proj}}
\def\ker{\operatorname{ker}}

\def\dim{\operatorname{dim}}

\def\grade{\operatorname{grade}}

\def\supp{\operatorname{Supp}}
\def\ann{\operatorname{Ann}}

\def\rel.{\operatorname{rel.}}
\def\degree{\operatorname{degree}}

\def\ZZ{\mathbb Z}
\def\QQ{\mathbb Q}
\def\fa{\operatorname{for~ all}}
\DeclareMathOperator{\homology}{H}
 
\DeclareMathOperator{\length}{length}
\DeclareMathOperator{\type}{type}
\newcommand{\NN}{\mathbb N}
\newcommand{\mm}{\mathfrak m}

\newcommand{\ov}{\overline}
\newcommand{\lm}{{\lambda}}

\newcommand{\R}{\mathcal R}

\newcommand{\I}{\underline I}
\newcommand{\n}{\underline{n}}
\newcommand{\p}{\underline {p}}

\newcommand{\m}{\underline {m}}
\newcommand{\rrr}{\underline{r}}
\newcommand{\ii}{i=1,\ldots,s}
\newcommand{\idl}{{I_1},\ldots,{I_s}}
\newcommand{\id}{{I_1}\cdots{I_s}}

\newcommand{\lf}{\left(}
\newcommand{\rg}{\right)}
\newcommand{\fil}{\mathcal F}
\newcommand{\gi}{{G_i(\fil)}_{++}}
\newcommand{\ggi}{{G_i(\fil)}}

\newcommand{\ga}{{G(\fil)}}

\newcommand{\yy}{{y_1},\ldots,{y_d}}

\newcommand{\po}{P_{\fil}}
\newcommand{\ho}{H_{\fil}}
\newcommand{\hf}{H_{\I}}
\newcommand{\pf}{P_{\I}}
\newcommand{\bl}{\begin{lemma}}
\newcommand{\el}{\end{lemma}}
\newcommand{\bt}{\begin{theorem}}
\newcommand{\et}{\end{theorem}}
\newcommand{\ben}{\begin{enumerate}}
\newcommand{\een}{\end{enumerate}}
\newcommand{\bpf}{\begin{proof}}
\newcommand{\eepf}{\end{proof}}
\newcommand{\beqn}{\begin{eqnarray*}}
\newcommand{\eeqn}{\end{eqnarray*}}
\newcommand{\beqnn}{\begin{eqnarray}}
\newcommand{\eeqnn}{\end{eqnarray}}
\newcommand{\bd}{\begin{definition}}
\newcommand{\ed}{\end{definition}}
\newcommand{\bp}{\begin{proposition}}
\newcommand{\ep}{\end{proposition}}
\newcommand{\bc}{\begin{corollary}}
\newcommand{\ec}{\end{corollary}}
\newcommand{\bex}{\begin{example}}
\newcommand{\eex}{\end{example}}

\newcommand{\wrt}{with respect to }
\newcommand{\CM}{Cohen-Macaulay }
\newcommand{\wlg}{ Without loss of generality }

\theoremstyle{plain}
\newtheorem{theorem}{Theorem}[section]
\newtheorem{corollary}[theorem]{Corollary}
\newtheorem{proposition}[theorem]{Proposition}
\newtheorem{lemma}[theorem]{Lemma}
\newtheorem{conjecture}[theorem]{Conjecture}

\newtheorem{example}[theorem]{Example}
\newtheorem{definition}[theorem]{Definition}

\theoremstyle{remark}

\numberwithin{equation}{theorem}
\title[Grothendieck-Serre Formula]
{Variations on the Grothendieck-Serre Formula for Hilbert Functions and their Applications}
\thanks{The first author is partially supported by a grant from Infosys Foundation}
\thanks{The second author is supported by CSIR Fellowship of Government of India}
\author{Shreedevi K. Masuti}
\address{Institute of Mathematical Sciences, CIT Campus, Taramani, Chennai, India-600113}
\email{shreedevikm@cmi.ac.in}
\author{Parangama Sarkar}
\address{Department of Mathematics, Indian Institute of Technology Bombay, Mumbai, 400076, India}
\email{parangama@math.iitb.ac.in}
\author{J. K. Verma}
\address{Department of Mathematics, Indian Institute
of Technology Bombay, Mumbai, India}
\email{jkv@math.iitb.ac.in}

\begin{document}
\begin{abstract}
In this expository paper we present proofs of Grothendieck-Serre formula for
multi-graded algebras and Rees algebras for admissible multi-graded filtrations. As applications, we derive formulas of Sally for postulation number of admissible filtrations and Hilbert coefficients. We also discuss a partial solution of Itoh's conjecture by Kummini and Masuti. We present an alternate proof of Huneke-Ooishi Theorem and a genaralisation for multi-graded filtrations. 

\end{abstract}
\maketitle
\section{Introduction}
The objective of this expository paper is to collect together several
fundamental results about Hilbert coefficients of admissible filtrations of
ideals which can be proved using various avatars of the Grothendieck-Serre
formula for the difference of the Hilbert function and Hilbert polynomial of a
finite graded module of a standard graded Noetherian ring.  The proofs presented
here provide a unified way of approaching these  results. Some of these results
are not known in the multi-graded case. We hope that the unified approach
presented here could lead to suitable multi-graded analogues of these results.

We begin by recalling the Grothendieck-Serre formula. For the sake of
simplicity, we assume that the graded rings considered in this paper are
standard and Noetherian.  Let $R=\bigoplus\limits_{n=0}^\infty R_n$ be a standard
graded Noetherian ring where $R_0$ is an Artinian local ring. Let
$M=\bigoplus\limits_{n\in \ZZ}M_n $ be a finite graded $R$-module of dimension $d.$ The
Hilbert function  of $M$ is the function $H(M,n)=\lm_{R_0} (M_n)$ for all $n \in
\ZZ.$  Here $\lm$ denotes the length function. Serre showed that there exists an integer $m$ so that $H(M,n)$ is given by a polynomial $P(M,x)\in \QQ[x]$ of degree $d-1$ such
that $H(M,n)=P(M,n) $  for all $n >  m.$ The smallest such $m$ is called the
postulation number of $M.$ Let
$R_+$ denote the homogeneous ideal of $R$ generated by elements of positive
degree and $[H^i_{R_+}(M)]_n$ denote the $n^{th}$ graded component of the 
$i^{th}$ local cohomology module  $H^i_{R_+}(M)$ of $M$ with respect to the
ideal $R_+.$ We put $\lm_{R_0}([H^i_{R_+}(M)]_n)=h^i_{R_+}(M)_n.$

\begin{theorem} {\em(Grothendieck-Serre)} For all $n \in \ZZ,$ we have
$$H(M,n)-P(M,n)=\sum_{i=0}^d(-1)^i h^i_{R_+}(M)_n.$$
\end{theorem}

\noindent
The GSF was proved in the fundamental paper \cite{serre} of J.-P. Serre. We quote from \cite{brod}:  ``In this paper, Serre introduced the theory of coherent sheaves over algebraic varieties over an algebraically closed field and a cohomology theory of such varieties with coefficients in coherent sheaves. He did speak of algebraic coherent sheaves, as at the first time he managed to introduce these theories with purely algebraic tools, using consequently the Zariski topology instead of the complex topology and homological methods instead of tools from multivariate complex analysis. 
Since then, the cohomology theory introduced in Serre's paper is often called Serre cohomology or sheaf cohomology.

One of the achievement of Serre's paper is the Grothendieck-Serre Formula, which is given there in terms of sheaf cohomology and showed in this way that sheaf cohomology gives a functorial understanding of the so called postulation problem of algebraic geometry, the problem which classically consisted in understanding the difference between the Hilbert function and the Hilbert-polynomial of the coordinate ring of a projective variety."

The Grothendieck-Serre Formula (GSF) is valid for non-standard graded rings also
if the Hilbert polynomial $P(M,x)$ is replaced by the Hilbert quasi-polynomial
\cite[Theorem 4.4.3]{BH}. The GSF has been generalized in several directions.
For some of the applications, we need it in the context of $\ZZ^s$-graded
modules over standard $\NN^s$-graded rings. In order to state the GSF for $\ZZ^s$-graded module, first we set up nptation and recall some definitions. Let $(R,\mm)$ be a Noetherian local ring and $\idl$ be $\mm$-primary ideals of $R.$ We  put  $e=(1,\ldots,1),\; \underline{0}= (0,\ldots,0)\in{\ZZ}^s$ and for all $\ii,$ $e_i=(0,\ldots,1,\ldots,0)\in{\ZZ}^s$ where $1$ occurs at $i$th position.  
 For $\n=(n_1,\ldots,n_s)\in{\ZZ}^s,$ we write ${\I}^{\n}=I_{1}^{n_1}\cdots I_{s}^{n_s}$ and $\n^+=(n_1^+,\ldots,n_s^+)$ where $ n_i^+=\max\{0,n_i\}$ for all $\ii.$ For $\alpha=({\alpha}_1,\ldots,{\alpha}_s)\in{\NN}^s,$ we put $|\alpha|={\alpha}_1+\cdots+{\alpha}_s.$ We define $\m=(m_1,\ldots,m_s)\geq\n=(n_1,\ldots,n_s)$ if $m_i\geq n_i$ for all $\ii.$ By the phrase ``for all large $\n$" we mean $\n\in\NN^s$ and $n_i\gg 0$ for all $\ii.$ For an $\NN^s$ (or a $\ZZ^s$)-graded ring $T$, the ideal generated by elements of degree $e$ is denoted by $T_{++}.$
\bd
A set of ideals $\fil=\lbrace\fil(\n)\rbrace_{\n\in \ZZ^s}$ is called a $\ZZ^s$-graded {\bf{$\I=(\idl)$-filtration}} if for all $\m,\n\in\ZZ^s,$
{\rm (i)} ${\I}^{\n}\subseteq\fil(\n),$
 {\rm (ii)} $\fil(\n)\fil(\m)\subseteq\fil(\n+\m)$  and {\rm (iii)} if $\m\geq\n,$ $\fil(\m)\subseteq\fil(\n)$.
\ed
Let $R=\bigoplus\limits_{\n\in\NN^s}R_{\n}$ be a standard Noetherian $\NN^s$-graded ring defined over a local ring $(R_{\underline 0},\mm)$ and $R_{++}=\bigoplus\limits_{\n\geq e}R_{\n}.$ Let ${\Proj}(R)$ denote the set of all homogeneous prime ideals $P$ in $R$ such that $R_{++}\nsubseteq P.$ For a finitely generated module $M,$ set $\supp_{++}(M)=\{P\in{\Proj}(R)\mid M_P\neq 0\}.$ Note that  $\supp_{++}(M)=V_{++}(\ann(M))$ \cite[Lemma 2.2.5]{gc}, \cite{hhrz}. 
\bd
The {\bf{relevant dimension}} of $M$ is
\[ \rel.\dim(M) = \left\{
  \begin{array}{l l}
   s-1  & \quad \text{if $\supp_{++}(M)=\emptyset$ }\\
    \max\{\dim\lf R/P\rg\mid P\in\supp_{++}(M)\} & \quad \text{if $\supp_{++}(M)\neq\emptyset.$ }
  \end{array} \right. \]
\ed
By \cite[Lemma 1.1]{hhrz}, $\dim\supp_{++}(M)=\rel.\dim(M)-s.$
M. Herrmann, E. Hyry, J. Ribbe and Z. Tang \cite[Theorem 4.1]{hhrz} proved that if $R=\bigoplus\limits_{\n\in\NN^s}R_{\n}$ is a standard Noetherian $\NN^s$-graded ring defined over an Artinian local ring $(R_{\underline 0},\mm)$ and $M=\bigoplus\limits_{\n\in\ZZ^s}M_{\n}$ is a finitely generated $\ZZ^s$-graded $R$-module then there exists a polynomial, called the Hilbert polynomial of $M, $ $P_M(x_1,x_2,\ldots,x_s) \in\QQ[x_1,\ldots,x_s]$ of total degree $\dim\supp_{++}(M)$ satisfying $P_M(\n)=\lm(M_{\n})$ for all large $\n.$ Moreover all monomials of highest degree in this polynomial have nonnegative coefficients.
\\The next two results are due to G. Colom\'{e}-Nin \cite[Propositions 2.4.2, 2.4.3]{gc} for non-standard multigraded rings. In section 2, we present her proofs to prove the same results for standard multi-graded rings for the sake of simplicity. These results were proved in the bigraded case by A.V. Jayanthan and J. K. Verma \cite{jv}.
\bp\label{vl}
Let $R=\bigoplus\limits_{\n\in\NN^s}R_{\n}$ be a standard Noetherian $\NN^s$-graded ring defined over a local ring $(R_{\underline 0},\mm)$ and $M=\bigoplus\limits_{\n\in\ZZ^s}M_{\n}$ a finitely generated $\ZZ^s$-graded $R$-module. Then
\ben
{
\item[(1)] For all $i\geq 0$ and $\n\in\ZZ^s$, $[H_{R_{++}}^i(M)]_{\n}$ is finitely generated 
$R_{\underline 0}$-module.
\item[(2)] For all large $\n$ and $i\geq 0,$ $[H_{R_{++}}^i(M)]_{\n}=0.$
}\een
\ep
\begin{theorem}{\em{(Grothendieck-Serre formula for $\ZZ^s$-graded modules)}} 
Let $R=\bigoplus\limits_{\n\in\NN^s}R_{\n}$ be a standard Noetherian
$\NN^s$-graded ring defined over an Artinian local ring $(R_{\underline 0},\mm)$
and $M=\bigoplus\limits_{\n\in\ZZ^s}M_{\n}$ a finitely generated $\ZZ^s$-graded
$R$-module. Let $H_M(\n)=\lm(M_{\n})$ and $P_M(x_1,\ldots,x_s)$ be the  the Hilbert polynomial of $M.$ Then for all $\n\in \ZZ^s,$ 
$$H_M(\n)-P_M(\n)= \sum\limits_{j\geq 0}(-1)^j h_{R_{++}}^j(M)_{\n}.$$
\end{theorem}

The above form of the GSF leads to another version of it which gives the difference between the Hilbert polynomial and the function  of $\ZZ^s$-graded filtrations of ideals in terms of local cohomology modules of various forms of Rees rings and associated graded rings of ideals. To define these,
let $t_1,t_2, \dots, t_s$ be indeterminates and 
${\underline{t}}^{\n}={t_1}^{n_1}\cdots{t_s}^{n_s}.$ We put 

\medskip 

\begin{center}
\begin{tabular}{ll}
  $\displaystyle\mathcal{R}(\fil)=\bigoplus\limits_{\n\in \NN^s}
{\fil}(\n){\underline{t}}^{\n}$  & {{the Rees ring of $\fil$}} , \\
 $\displaystyle\mathcal{R}'(\fil)=\bigoplus\limits_{\n\in{\ZZ}^s}{\fil}(\n){\underline{t}}^{\n}$ 
 & {{the extended Rees ring of $\fil$}},  \\
 $\ga=\displaystyle\bigoplus\limits_{\n\in{\NN}^s}\frac{{\fil}(\n)}{{\fil}(\n+e)}$ &
   {{the associated  multi-graded ring of $\fil$ with respect to $\fil(e)$}}, \\
    
   $\displaystyle\ggi=\bigoplus\limits_{\n\in{\NN}^s}\frac{{\fil}(\n)}{{\fil}(\n+e_i)}$ 
   &   the associated graded ring of $\mathcal F$ with respect to $\mathcal F(e_i).$
 \end{tabular}
 \end{center}
For $\fil=\{\I^{\n}\}_{\n \in \ZZ^s}$, we set $\mathcal R(\fil)=\mathcal R(\I)$ and $\mathcal R^\prime(\fil)=
\mathcal R^\prime(\I),$ $G(\fil)=G(\I)$ and $G_i(\fil)=G_i(\I)$ for all $\ii.$
 
\bd A $\ZZ^s$-graded $\I$-filtration $\fil=\lbrace\fil(\n)\rbrace_{\n\in \ZZ^s}$ of ideals in $R$ is 
called an $\I$-{\bf{admissible filtration}} if 
${{\fil}(\n)}={\fil}(\n^+)$ and $\mathcal{R}'(\fil)$
 is a  finite $\mathcal{R}'(\I)$-module. For $s=1,$ if a filtration $\fil$ is $I$-admissible for some $\mm$-primary ideal $I$ then it is also $I_1$-admissible.\ed
 
Primary examples of $\I$-admissible filtrations are $\{\I^{\n}\}_{\n \in \ZZ^s}$ in a Noetherian local ring and $\{\ov{\I^{\n}}\}_{\n \in \ZZ^s}$ in an analytically unramified local ring. Recall that for an ideal $I$ in $R,$ the integral closure of $I$ is the ideal 
\begin{eqnarray*}
 \ov{I}:=\{x \in R \mid x^n+a_1x^{n-1}+\cdots+a_{n-1}x+a_n=0 \mbox{ for some } n \in \NN \mbox{ and } a_i \in I^i \text{ for  } i=1,2\ldots,n \}.
 \end{eqnarray*}

We now set up the notation for a variety of Hilbert polynomials associated to filtrations of ideals.
Let $I$  be an $\mm$-primary ideal of a Noetherian local ring $(R,\mm)$ of dimension $d.$ 
 For a $\ZZ$-graded $I$-admissible filtration $\mathcal I=\{I_n\}_{n\in\ZZ},$ Marley \cite{marleythesis} proved existence of a polynomial $P_{\mathcal I}(x)\in\QQ[x]$ of degree $d,$ written in the form, $$P_{\mathcal I}(n)=e_0(\mathcal I)\binom{n+d-1}{d}-e_1(\mathcal I)
\binom{n+d-2}{d-1}+\cdots+(-1)^d e_d(\mathcal I)
$$ such that $P_{\mathcal I}(n)=H_{\mathcal I}(n)$ for all large $n,$ where $H_{\mathcal I}(n)=\lm (R/I_n)$ is the {\bf Hilbert function} of the filtration $\mathcal I.$ The coefficients $e_i(\mathcal I)$ for $i=0,1,\ldots,d$ are integers, called the {\bf Hilbert coefficients} of $\mathcal I.$ The coefficient $e_0(\mathcal I)$ is called the multiplicity of $\mathcal I.$ P. Samuel \cite{samuel} showed existence of this polynomial for the $I$-adic filtration $\{I^n\}_{n\in\ZZ}.$ Many results about Hilbert polynomials for admissible filtrations were proved in \cite{GR} and \cite{RV}.
\\For $\mm$-primary ideals $\idl,$ B. Teissier \cite{T} proved that  for all $\n$ sufficiently large, the {\bf Hilbert function} $\hf(\n)=\lm\lf{R}/{{\I}^{\n}}\rg$ coincides with a polynomial $$\pf(\n)=\displaystyle\sum\limits_{\substack{\alpha=({\alpha}_1,\ldots,{\alpha}_s)\in{\NN}^s \\ |\alpha|\leq d}}(-1)^{d-|{\alpha}|}{e_{\alpha}}(\I)\binom{{n_1}+{{\alpha}_1}-1}{{\alpha}_1}\cdots\binom{{n_s}+{{\alpha}_s}-1}{{\alpha}_s}$$ of degree $d$, called the {\bf{Hilbert polynomial}} of $\I.$ Here we assume that $s \geq 2$ in order to write  $\pf(\n)$ in the above form. This was proved by P. B. Bhattacharya for $s=2$ in \cite{B}. Here $e_{\alpha}(\I)$ are integers which are called the {\bf{Hilbert coefficients}} of $\I.$ D. Rees \cite{rees3} showed that ${e_{\alpha}}(\I)>0$ for $|\alpha|=d.$ These are called the 
{\bf{mixed multiplicities}} of $\I.$ 
\\For an $\I$-admissible filtration $\fil=\lbrace\fil(\n)\rbrace_{\n\in \ZZ^s}$ in a Noetherian local ring $(R,\mm)$ of dimension $d,$ Rees \cite{rees3} showed the existence of a polynomial 
$$\po(\n)=\displaystyle\sum\limits_{\substack{\alpha=({\alpha}_1,\ldots,{\alpha}_s)\in{\NN}^s \\ 
|\alpha|\leq d}}(-1)^{d-|{\alpha}|}{e_{\alpha}}(\fil)\binom{{n_1}+{{\alpha}_1}-1}{{\alpha}_1}
\cdots\binom{{n_s}+{{\alpha}_s}-1}{{\alpha}_s}$$ of degree $d$ which coincides with the {\bf{Hilbert function}} $\ho(\n)=\lm\lf{R}/{{\fil}(\n)}\rg$ for all large $\n$ \cite{rees3}. This polynomial is called the {\bf{Hilbert polynomial}} of $\fil.$ Rees \cite[Theorem 2.4]{rees3} proved that $e_{\alpha}(\fil)=e_{\alpha}(\I)$ for all $\alpha\in\NN^s$ such that $|\alpha|=d.$ 

In Section 2, we prove the following version of the GSF for the extended Rees algebras. It was proved for $I$-adic filtration and  for nonnegative integers  by Johnston-Verma \cite{1jv} and for $\ZZ$-graded admissible filtration of ideals by
C. Blancafort for all integers \cite{blancafort}.

\begin{theorem}\cite[Theorem 4.3]{msv}
Let $(R,\mm)$ be a Noetherian local ring of dimension $d$ and $\idl$ be $\mm$-primary ideals of $R.$ Let $\fil=\lbrace\fil(\n)\rbrace_{\n\in{\ZZ}^s}$ be an $\I$-admissible filtration of ideals in $R.$ Then
\ben
{\item[(1)] $h_{\R_{++}}^i(\mathcal{R}'(\fil))_{\n}<\infty$ for all $i\geq 0$ and $\n\in\ZZ^s.$ 
\item[(2)] $\po(\n)-\ho(\n)=\sum\limits_{i\geq 0}(-1)^i h_{\R_{++}}^i(\mathcal{R}'(\fil))_{\n}$ 
for all $\n\in\ZZ^s.$}\een
\end{theorem}

In section 3, we derive explicit formulas in terms of the Ratliff-Rush closure
filtration of a multi-graded filtration of ideals  for the graded components of
the local cohomology modules of certain Rees rings and associated graded rings.
For an ideal $I$ in a Noetherian ring $R,$ L. J. Ratliff and D. Rush \cite{rr}
introduced the ideal $$\tilde I=\bigcup_{k \geq 1}(I^{k+1}:I^k),$$ called the
{\bf Ratliff-Rush closure }of $I.$ If $I$ has a regular element then the ideal $\tilde I$ has some nice properties
such as for all large $n,$ $(\tilde I)^n=I^n, \tilde{I^n}=I^n$ etc. If $I$ is an
$\mm$-primary regular ideal then $\tilde I$ is the largest ideal with respect to
inclusion having the same Hilbert polynomial as that of $I.$ Blancafort 
\cite{blancafort} introduced Ratliff-Rush 
closure filtration of an $\mathbb N$-graded good filtration. Let $(R,\mm)$ be a
Noetherian local ring and $\idl$ be $\mm$-primary ideals of $R.$ Let
$\fil=\{\fil(\n)\}_{\n\in\ZZ^s}$ be an $\I$-admissible filtration of ideals in $R.$ We need the concept of the Ratliff-Rush closure of $\fil$ in order to find formulas for certain local cohomology modules.
\bd
The {\it{\bf{Ratliff-Rush closure filtration}}} 
 of $\fil=\lbrace\fil(\n)\rbrace_{\n\in \ZZ^s}$ is the filtration of ideals
$\breve{\fil}=\{\breve \fil(\n)\}_{\n\in\ZZ^s}$ 
 given by
 \ben
 \item[(1)] $\breve \fil(\n)=\bigcup\limits_{k\geq 1} (\fil(\n+ke):\fil(e)^k)$
for all $\n\in\NN^s,$
 \item[(2)] $\breve \fil(\n)=\breve\fil(\n^+)$ for all $\n\in\ZZ^s.$
 \een
\ed
The next three results to be proved in section 3 are needed to prove several results about Hilbert coefficients in section 5. 
\bp\cite[Proposition 3.5]{msv}
Let $(R,\mm)$ be a Cohen-Macaulay local ring of dimension two with infinite
residue field and $\idl$ be $\mm$-primary ideals in $R.$ Let  
$\fil=\lbrace\fil(\n)\rbrace_{\n\in{\ZZ}^s}$ be an $\I$-admissible filtration of ideals in $R.$ 
Then for all $\n\in\NN^s,$ 
$$[H_{\R(\fil)_{++}}^1(\R(\fil))]_{\n}\cong\frac{\breve\fil(\n)}{\fil(\n)}.$$

\ep
\bp\cite[Theorem 3.5]{blancafort}
Let $(R,\mm)$ be a Cohen-Macaulay local ring of dimension two with infinite
residue field, $I$ an $\mm$-primary ideal of $R$ and $\fil=\lbrace I_n\rbrace_{n\in{\ZZ}}$ be an $I$-admissible
filtration of ideals in $R.$ 
Then \[ [H_{\R(\fil)_{+}}^1(\R^\prime(\fil))]_{n} = \left\{
  \begin{array}{l l}
   \breve{I_n}/I_n  & \quad \text{if $n\geq 0$ }\\
    0 & \quad \text{if $n<0.$ }
  \end{array} \right. \] 
\ep

\begin{theorem}\cite[Theorem 3.3]{msv}
Let $(R,\mm)$ be a Noetherian local ring of dimension $d\geq 1$ with infinite
residue field and $\idl$ be $\mm$-primary ideals in $R$ such that
grade$(\id)\geq 1.$ Let $\fil=\lbrace\fil(\n)\rbrace_{\n\in \ZZ^s}$ be an $\I$-
admissible filtration of ideals in  $R.$ Then for all $\n\in \NN^s$ and $\ii,$ 
$$[H^0_{G_i(\mathcal F)_{++}}(G_i(\mathcal
F))]_{\n}=\frac{\breve\fil(\n+e_i)\cap\fil(\n)}{\fil(\n+e_i)}.
$$ \et

In section 4, we present several applications of the GSF for Rees algebra and associated graded ring of an ideal.  The first
application due to J. D. Sally, who pioneered these techniques for the study of
Hilbert-Samuel coefficients, shows the connection of the postulation number
with reduction number. Let $(R,\mm)$ be a Noetherian local ring, $I$ be
an $\mm$-primary ideal and $\fil=\lbrace I_n\rbrace_{n\in\ZZ}$ be an admissible $I$-filtration of ideals
in $R.$
\bd
A {\bf{reduction}} of an $I$-admissible filtration $\mathcal F=\lbrace I_n\rbrace_{n\in\ZZ}$ is an ideal $J\subseteq I_1$ such that $JI_{n} = I_{n+1}$ for all large $n.$ A {\bf{minimal reduction}} of $\mathcal F$ is a reduction of $\mathcal F$  minimal with respect to inclusion. For a minimal reduction $J$ of $\mathcal F,$ we set $${r_J}(\mathcal F)=\min\lbrace m:JI_n=I_{n+1}
\mbox{ for }n\geq m\rbrace\mbox{ and }{r}(\mathcal F)=\min\lbrace {r_J}(\mathcal I):J\mbox{ is a minimal reduction of }\mathcal F \rbrace.$$ For $\fil=\lbrace I^n\rbrace_{n\in\ZZ},$ we set $r_J(\fil)=r_J(I)$ and $r(\fil)=r(I).$\ed
\bd
An integer $n\in\ZZ$ is called the {\bf{postulation number}} of $\fil,$ denoted by $n(\fil),$ if
$\po(m)=\ho(m)$ for all $m>n$ and $\po(n)\neq\ho(n).$ It is denoted by
$n(\fil).$
\ed
The next result was proved by J. D. Sally \cite{sa} for the $\mm$-adic filtration. Her proof remains valid for any admissible filtration.
\bt 
Let $(R,\mm)$ be a \CM local ring of dimension $d\geq 1$ with infinite residue
field, $I$ an $\mm$-primary ideal and $\fil=\{I_n\}$ be an $I$-admissible
filtration of ideals in $R.$ Let $H_R(n)=\lm\lf{I_n}/{I_{n+1}}\rg$ and
$P_R(X)\in \QQ[X]$ such that $P_R(n)=H_R(n)$ for all large $n.$ Suppose $\grade
G(\fil)_+\geq d-1.$ Then for a minimal reduction $J=(x_1,\ldots, x_d)$ of
$\fil,$ $H_R(r_J(\fil)-d)\neq P_R(r_J(\fil)-d)$ and $H_R(n)= P_R(n)$ for all
$n\geq r_J(\fil)-d+1.$ 
\et
The following result is due to Marley \cite[Corollary 3.8]{marleythesis}. We
give another proof which follows from the above theorem.
\bt\cite[Corollary 3.8]{marleythesis}
Let $(R,\mm)$ be a \CM local ring of dimension $d\geq 1$ with infinite residue
field, $I$ an $\mm$-primary ideal and $\fil=\{I_n\}$ be an $I$-admissible filtration
of ideals in $R.$ Let $\grade G(\fil)_+\geq d-1.$ Then $r(\fil)=n(\fil)+d.$
\et
In section 5, we discuss several results about non-negativity of Hilbert
coefficients of multi-graded filtrations of ideals as easy consequences of the
GSF for such filtrations. We prove the following result which implies earlier
results of Northcott, Narita and Marley.

\bt\cite[Theorem 5.6]{msv}
Let $(R,\mm)$ be a Cohen-Macaulay local ring of dimension $d\geq 1$ and $\idl$
be $\mm$-primary ideals of $R.$ Let $\fil=\lbrace\fil(\n)\rbrace_{\n\in\ZZ^s}$
be an $\I$-admissible filtration of ideals in $R.$ Then 
\ben{
\item[(1)] $e_{\alpha}(\fil)\geq 0$ where
$\alpha=({\alpha}_1,\ldots,{\alpha}_s)\in\NN^s,$ 
$|\alpha|\geq d-1.$
\item[(2)] $e_{\alpha}(\fil)\geq 0$ where
$\alpha=({\alpha}_1,\ldots,{\alpha}_s)\in\NN^s,$ $|\alpha|= d-2$ and $d\geq
2.$}\een
\et

We also discuss the results of S. Itoh about non-negativity and vanishing of
the third coefficient of the normal Hilbert polynomial of the filtration
$\{\ov{I^n}\}_{n\in\ZZ}$ in an analytically unramified Cohen-Macaulay local ring. We prove an analogue of a theorem due to Sally for
admissible filtrations in two-dimensional Cohen-Macaulay local rings which gives
explicit formulas for all the coefficients
of their Hilbert polynomial. Here again we show that these formulas follow in a
natural way from the  variant of GSF for Rees algebra of the filtration.
\bp
Let $(R,\mm)$ be a two-dimensional \CM local ring, $I$ be any $\mm$-primary
ideal of $R$ and $\fil=\lbrace I_n\rbrace_{n\in{\ZZ}}$ an admissible
$I$-filtration of ideals in $R.$ Then 
\ben{
\item $\displaystyle{\lm\lf H_{\R(\fil)_+}^2(\R(\fil)_0\rg=e_2(\fil)},$
\item $\displaystyle{\lm\lf
H_{\R(\fil)_+}^2(\R(\fil))_1\rg=e_0(\fil)-e_1(\fil)+e_2(\fil)-\lm\lf\frac{R}{
\breve {I_1}}\rg},$
\item $\displaystyle{\lm\lf
H_{\R(\fil)_+}^2(\R(\fil))_{-1}\rg=e_1(\fil)+e_2(\fil)}.$}\een 
\ep

C. Huneke \cite{huneke} and A. Ooishi \cite{ooishi} independently proved that if $(R,\mm)$ is a \CM local ring of dimension $d\geq 1$ and $I$ is an $\mm$-primary ideal then $e_0(I)-e_1(I)=\lambda(R/I)$ if and only if $r(I)\leq 1.$ Huckaba and Marley \cite{huckaba-marley} proved this result for $\ZZ$-graded admissible filtrations. In Section 6, we present a proof, due to Blancafort, for $\ZZ$-graded admissible filtrations of Huneke-Ooishi Theorem. The original proofs due to Huneke and Ooishi did not employ local
cohomology and relied on use of superficial sequences. Our purpose in
presenting the alternative proof using the GSF for Rees algebras is to motivate
the proof of an analogue of the Huneke-Ooishi Theorem for multi-graded
filtrations of ideals. 

\bt\cite[Theorem 4.3.6]{blancafort thesis}
Let $(R,\mm)$ be a \CM local ring with infinite residue field of dimension $d\geq 1,$ $I_1$ an $\mm$-primary ideal and $\fil=\lbrace I_n\rbrace_{n\in{\ZZ}}$ be an $I_1$-admissible filtration of ideals in $R.$ Then the following are equivalent:
\ben{
\item $e_0(\fil)-e_1(\fil)=\lm\lf{R}/{I_1}\rg,$
\item $r(\fil)\leq 1.$
}\een
In this case, $e_2(\fil)=\cdots=e_d(\fil)=0,$ $G(\fil)$ is Cohen-Macaulay, $ n(\fil) \leq 0,$ $r(\fil)$ is independent of the reduction chosen and  $\fil=\lbrace I_1^n\rbrace .$

\et

Using the GSF for multi-graded Rees algebras we prove the following analogue of the Huneke-Ooishi Theorem for multi-graded admissible filtrations.
\begin{theorem}\cite[Theorem 5.5]{msv}
Let $(R,\mm)$ be a Cohen-Macaulay local ring of dimension $d\geq 1$ and $\idl$
be $\mm$-primary ideals of $R.$ Let $\fil=\lbrace\fil(\n)\rbrace_{\n\in \ZZ^s}$
be an $\I$-admissible filtration of ideals in $R.$ Then for all $\ii,$\\
$(1)$ $e_{(d-1)e_i}(\mathcal F) \geq e_1(\mathcal F^{(i)}),$\\
 $(2)$ $e(I_i) - e_{(d-1)e_i}(\mathcal F) \leq \lm(R/\mathcal F{(e_i)}),$\\
 $(3)$ $e(I_i) - e_{(d-1)e_i}(\mathcal F) = \lm(R/\mathcal F{(e_i)})$
 if and only if 
$r(\mathcal F^{(i)}) \leq 1$ and $e_{(d-1)e_i}(\mathcal F)=e_1(\mathcal
F^{(i)}).$
\end{theorem}

The vanishing of the constant term of the Hilbert polynomial of a filtration gives insight into the filtration as well as the local ring. 
 For any $\mm$-primary ideal $I$ in an analytically unramified local ring $(R,\mm)$
of dimension $d,$  the {\bf normal Hilbert  function} of $I$ is defined to be the function $\ov{H}(I,n)=\lm(R/\ov{I^n}).$ Rees showed that for large $n,$ it is given by the {\bf normal Hilbert polynomial}  
$$
\ov{P}(I,x)=\ov{e}_0(I)\binom{x+d-1}{d} -\ov{e}_1(I) \binom{x+d-2}{d-1}+\cdots+(-1)^d\ov{e}_d(I).$$
The integers $\ov{e}_0(I), \ov{e}_1(I), \ldots, \ov{e}_d(I)$ are called the
{\bf normal Hilbert coefficients} of $I.$ Rees defined a   $2$-dimensional normal analytically unramified  local ring $(R,\mm)$  to be {\bf pseudo-rational} if $\ov{e}_2(I)=0$ for all  $\mm$-primary ideals. It can be shown that two-dimensional local rings having a rational singularity are pseudo-rational. It is natural to characterise $\ov{e}_2(I)=0$ in terms of computable data. This was 
considered by Huneke \cite{huneke} in which he proved

 \bt \cite[Theorem 4.5]{huneke}
Let $(R,\mm)$ be a two-dimensional analytically unramified Cohen-Macaulay local
ring. Let $I$ be an  $\mm$-primary ideal.  Then $\ov{e}_2(I)=0$ if and only if  $\ov{I^n}=(x,y)\ov{I^{n-1}}$ for $n \geq 2$ and for any minimal reduction $(x,y)$ of $I.$
\et

A similar result was proved by Itoh \cite{itoh} about vanishing of ${\ov e}_2(I).$ Using the GSF for multi-graded filtrations, we prove the following theorem which  characterises   the vanishing of the constant term of the Hilbert polynomial of a multi-graded admissible filtration and derive results of Itoh and Huneke as consequences.
\bt
\cite[Theorem 5.7]{msv}
Let $(R,\mm)$ be a Cohen-Macaulay local ring of dimension two and $\idl$ be $\mm$-primary ideals of $R.$ Let $\fil=\lbrace\fil(\n)\rbrace_{\n\in\ZZ^s}$ be an $\I$-admissible filtration of ideals in $R.$ Then
$e_{\underline 0}(\fil)=0$ implies $e(I_i)-e_{e_i}(\fil)=\lm\lf\frac{R}{\breve\fil(e_i)}\rg$ for all $\ii.$ Suppose $\breve\fil$ is $\I$-admissible filtration, then the converse is also true.
 \et
 
 {\bf Acknowledgement:} We thank Professor Markus Brodmann for discussions.
 
\section{Variations on the Grothendieck-Serre formula}
The main aim of this section is to prove the Grothendieck-Serre formula (Theorem \ref{GS}) and its variations. In \cite[Propositions 2.4.2, 2.4.3]{gc}, Colom\'{e}-Nin proved the Grothendieck-Serre formula  for non-standard multigraded rings. For the sake of simplicity, we present her proof for standard multi-graded rings. As a consequence we prove \cite[Theorem 4.3]{msv} (Theorem \ref{r3}) which relates the difference of Hilbert polynomial and Hilbert function of an ${\I}$-admissible filtration to the Euler characteristic of the extended multi-Rees algebra. 

We recall the following Lemma from \cite{gc} which is needed to prove Theorem \ref{GS}.

\bl\cite[Lemma 2.2.8]{gc}\label{reduce}
Let $R=\bigoplus\limits_{\n\in\NN^s}R_{\n}$ be a standard Noetherian $\NN^s$-graded ring defined over a local ring $(R_{\underline 0},\mm)$ and $M=\bigoplus\limits_{\n\in\ZZ^s}M_{\n}$ a finitely generated $\ZZ^s$-graded $R$-module. Let $x\in R_{\n}$ where $\n \geq e$ and $x\notin \bigcup\limits_{P\in Ass(M)}P.$ Then $\rel.\dim(M/xM)=\rel.\dim(M)-1.$
\el

\bp\label{vl}
Let $R=\bigoplus\limits_{\n\in\NN^s}R_{\n}$ be a standard Noetherian $\NN^s$-graded ring defined over a local ring $(R_{\underline 0},\mm)$ and $M=\bigoplus\limits_{\n\in\ZZ^s}M_{\n}$ a finitely generated $\ZZ^s$-graded $R$-module. Then
\ben
{
\item[(1)] For all $i\geq 0$ and $\n\in\ZZ^s$, $[H_{R_{++}}^i(M)]_{\n}$ is finitely generated 
$R_{\underline 0}$-module.
\item[(2)] For all large $\n$ and $i\geq 0,$ $[H_{R_{++}}^i(M)]_{\n}=0.$
}\een
\ep
\bpf  Note that $R_{++}$ is finitely generated. We prove both $(1)$ and $(2)$ together by induction on $i.$ Suppose $i=0.$ Note that $H_{R_{++}}^0(M)\subseteq M$ and hence $H_{R_{++}}^0(M)$ is finitely generated $R$-module. Let $\{\gamma_1,\dots,\gamma_q\}$ be a generating set of $H_{R_{++}}^0(M)$ as an $R$-module and $deg(\gamma_j)=p(j)=(p(j1),\dots,p(js))$ for all $j=1,\dots,q.$ Let $\alpha_i=\max\{|p(ji)|:j=1,\dots,q\}$ for all $i=1,\dots,s$ and $\alpha=(\alpha_1,\dots,\alpha_s).$ Since $H_{R_{++}}^0(M)$ is $R_{++}$-torsion, there exists an integer $t\geq 1$ such that $R_{++}^tH_{R_{++}}^0(M)=0.$ Then for all $\n\geq \alpha+te,$ $$[H_{R_{++}}^0(M)]_{\n}=R_{\n-p(1)}\gamma_1+\dots+R_{\n-p(q)}\gamma_q\subseteq  R_{++}^tH_{R_{++}}^0(M)=0.$$
\\Fix $\n\in\ZZ^s.$ Since $R$ is a standard Noetherian $\NN^s$-graded ring defined over $R_{\underline 0},$ there exist
elements $a_{i1},\ldots, a_{ik_i}\in R_{e_i}$ for all $\ii$ such that each nonzero element of $[H_{R_{++}}^0(M)]_{\n}$ can be written as sum of monomials $\prod\limits_{1\leq i\leq s} a_{i1}^{t_{i1}}\cdots a_{ik_i}^{t_{ik_i}}\gamma_j$ of degree $\n$ with coefficients from $R_{\underline 0}$ where $j=1,\dots,q,$ $t_{i1},\ldots,t_{ik_i}\geq 0.$ Since $0\leq t_{i1},\ldots,t_{ik_i}\leq n_i-p(ji),$ the number of monomial generators are finite. Hence   
$[H_{R_{++}}^0(M)]_{\n}$ is finitely generated $R_{\underline 0}$-module.
\\ Now assume $i>0.$ Let $M^\prime$ denote $\displaystyle{{M}/{H_{R_{++}}^0(M)}}.$ Consider the short exact sequence of $R$-modules
$$0\longrightarrow H_{R_{++}}^0(M)\longrightarrow M\longrightarrow M^\prime\longrightarrow 0$$ which gives long exact sequence of local cohomology modules
$$\cdots\longrightarrow H_{R_{++}}^i(H_{R_{++}}^0(M))\longrightarrow H_{R_{++}}^i(M)\longrightarrow H_{R_{++}}^i(M^\prime)\longrightarrow\cdots.$$ Since $H_{R_{++}}^0(M)$ is $R_{++}$-torsion, $H_{R_{++}}^i(H_{R_{++}}^0(M))=0$ for all $i\geq 1.$ Thus \beqnn\label{own} H_{R_{++}}^i(M)&\simeq& H_{R_{++}}^i(M^\prime)\mbox{ for all }i\geq 1. \eeqnn
By \cite[Lemma 2.4.1]{gc}, there exists an element $x\in R_{\p}$ for some $\p\geq e$ such that $x\notin P$ for all $P\in Ass(M^\prime)=Ass(M)\setminus V(R_{++}).$ Fix $i\geq 1.$ Consider the short exact sequence of $R$-modules $$0\longrightarrow M^\prime(-\p)\overset{.x}\longrightarrow M^\prime\longrightarrow M^\prime/xM^\prime\longrightarrow 0$$ which gives long exact sequence of local cohomology modules whose $\rrr$th component is {\beqn\displaystyle\cdots\longrightarrow \left[H_{R_{++}}^{i-1}\lf M^\prime/xM^\prime\rg\right]_{\rrr}\longrightarrow\left[H_{R_{++}}^i(M^\prime)\right]_{\rrr-\p}\overset{.x}\longrightarrow \left[H_{R_{++}}^i(M^\prime)\right]_{\rrr}\longrightarrow \left[H_{R_{++}}^i\lf  M^\prime/xM^\prime\rg\right]_{\rrr}\longrightarrow\cdots.\eeqn} By inductive hypothesis $\left[H_{R_{++}}^{i-1}\lf M^\prime/xM^\prime\rg\right]_{\m}=0$ for all large $\m,$ say, for all $\m\geq \underline k$ for some $\underline k\in \NN^s.$ Then for all $\n\geq \underline k,$ 
we have the exact sequence
$$0\longrightarrow\left[H_{R_{++}}^i(M^\prime)\right]_{\n-\p}\overset{.x}\longrightarrow\left[H_{R_{++}}^i(M^\prime)\right]_{\n}.$$ Since $H_{R_{++}}^i(M^\prime)$ is $R_{++}$-torsion and $x\in R_{++}$, we have $\left[H_{R_{++}}^i(M^\prime)\right]_{\m}=0$ for all $\m\geq \underline k-\p.$ Hence we prove part $(2).$
\\Fix $i>0$ and $\n\in\ZZ^s.$ By \cite[Lemma 2.4.1]{gc}, there exists an element $y\in R_{++}$ such that $y\notin P$ for all $P\in Ass(M^\prime)=Ass(M)\setminus V(R_{++})$ and we can assume $\degree(y)=\m$ such that $[H_{R_{++}}^i(M')]_{\rrr}=0$ for all $\rrr\geq\n+\m.$ Consider the short exact sequence of $R$-modules $$0\longrightarrow M'(-\m)\overset{.y}\longrightarrow M'\longrightarrow M'/yM'\longrightarrow 0$$ which gives long exact sequence of cohomology modules whose $\m+\n$-th component is \beqn\displaystyle\cdots\longrightarrow \left[H_{R_{++}}^{i-1}\lf M'/yM'\rg\right]_{\m+\n}\longrightarrow\left[H_{R_{++}}^i(M')\right]_{\n}\overset{.y}\longrightarrow \left[H_{R_{++}}^i(M')\right]_{\m+\n}\longrightarrow \cdots.\eeqn
Since $[H_{R_{++}}^i(M')]_{\m+\n}=0$ and by induction hypothesis $\left[H_{R_{++}}^{i-1}\lf M'/yM'\rg\right]_{\m+\n}$ is finitely generated $R_{\underline 0}$-module, from the above exact sequence, we get $[H_{R_{++}}^i(M')]_{\n}$ is finitely generated $R_{\underline 0}$-module. Hence by equation (\ref{own}), we get the required result.
\eepf
\begin{theorem}{\em{(Grothendieck-Serre formula for multi-graded modules)}}\label{GS} 
Let $R=\bigoplus\limits_{\n\in\NN^s}R_{\n}$ be a standard Noetherian $\NN^s$-graded ring defined over an Artinian local ring $(R_{\underline 0},\mm)$ and $M=\bigoplus\limits_{\n\in\ZZ^s}M_{\n}$ a finitely generated $\ZZ^s$-graded $R$-module. Let $H_M(\n)=\lm(M_{\n})$ and $P_M(x_1,\ldots,x_s)$ be the Hilbert polynomial of $M.$ Then for all $\n\in \ZZ^s,$ $$H_M(\n)-P_M(\n)=
\sum\limits_{j\geq 0}(-1)^j h_{R_{++}}^j(M)_{\n}.$$
\end{theorem}
\bpf
For all $\n\in \ZZ^s,$ we define $\chi_M(\n)=\sum\limits_{j\geq 0}(-1)^jh_{R_{++}}^j(M)_{\n}$ and $f_M(\n)=H_M(\n)-P_M(\n).$ We use induction on $\rel.\dim(M).$ Suppose $\rel.\dim(M)=s-1.$ Then $\supp_{++}(M)=V_{++}(\ann(M))=\emptyset.$ Therefore there exists an integer $k\geq 1$ such that ${R_{++}^k}M=0.$ Hence $H_{R_{++}}^0(M)=M$ and $H_{R_{++}}^i(M)=0$ for all $i\geq 1.$ Since $P_M(X_1,\ldots,X_s)$ has degree $-1,$ we have $P_M(\n)=0$ for all $\n\in \ZZ^s.$ Thus we get the required equality.
\\ Assume that $\rel.\dim(M)\geq s.$ Let $M^\prime$ denote $\displaystyle{{M}/{H_{R_{++}}^0(M)}}.$ Consider the short exact sequence of $R$-modules
$$0\longrightarrow H_{R_{++}}^0(M)\longrightarrow M\longrightarrow M^\prime\longrightarrow 0$$ which gives long exact sequence of local cohomology modules
$$\cdots\longrightarrow H_{R_{++}}^i(H_{R_{++}}^0(M))\longrightarrow H_{R_{++}}^i(M)\longrightarrow H_{R_{++}}^i(M^\prime)\longrightarrow\cdots.$$  Note that $H_{R_{++}}^0(M)$ is $R_{++}$-torsion. Hence for all $i\geq 1,$ $H_{R_{++}}^i(H_{R_{++}}^0(M))=0$ and \beqnn\label{equal} H_{R_{++}}^i(M)\simeq H_{R_{++}}^i(M^\prime).\eeqnn  Since $H_M(\n)=H_{M^\prime}(\n)+h_{R_{++}}^0(M)_{\n}$ and hence by Proposition \ref{vl} part $(2),$ $P_M(\n)=P_{M^\prime}(\n).$ Thus
$$H_M(\n)-P_M(\n)= H_{M^\prime}(\n)+h_{R_{++}}^0(M)_{\n}-P_{M^\prime}(\n)= H_{M^\prime}(\n)-P_{M^\prime}(\n)+h_{R_{++}}^0(M)_{\n}.$$ Therefore by the equation (\ref{equal}), it is enough to prove the result for $M^\prime.$
By \cite[Lemma 2.4.1]{gc}, there exists an element $x\in R_{\p}$ for some $\p\geq e$ such that $x\notin P$ for all $P\in Ass(M^\prime)=Ass(M)\setminus V(R_{++}).$ Consider the short exact sequence of $R$-modules $$0\longrightarrow M^\prime(-\p)\overset{.x}\longrightarrow M^\prime\longrightarrow M^\prime/xM^\prime\longrightarrow 0$$ which gives long exact sequence of cohomology modules whose $\rrr$th component is $$\cdots\longrightarrow \left[H_{R_{++}}^{i-1}\lf M^\prime/xM^\prime\rg\right]_{\rrr}\longrightarrow\left[H_{R_{++}}^i(M^\prime)\right]_{\rrr-\p}\overset{.x}\longrightarrow \left[H_{R_{++}}^i(M^\prime)\right]_{\rrr}\longrightarrow \left[H_{R_{++}}^i\lf  M^\prime/xM^\prime\rg\right]_{\rrr}\longrightarrow\cdots.$$ Thus for all $\n\in \ZZ^s,$ $H_{M^\prime/xM^\prime}(\n)=H_{M^\prime}(\n)-H_{M^\prime}(\n-\p).$ Hence $P_{M^\prime/xM^\prime}(\n)=P_{M^\prime}(\n)-P_{M^\prime}(\n-\p).$ By Lemma \ref{reduce}, $\rel.\dim(M^\prime/xM^\prime)<\rel.\dim(M^\prime).$ Therefore for all $\n\in \ZZ^s,$ $$
f_{M^\prime}(\n)-f_{M^\prime}(\n-\p)= f_{M^\prime/xM^\prime}(\n)= \chi_{M^\prime/xM^\prime}(\n)=\chi_{M^\prime}(\n)-\chi_{M^\prime}(\n-\p).$$ Hence $f_{M^\prime}(\n)-\chi_{M^\prime}(\n)=f_{M^\prime}(\n-\p)-\chi_{M^\prime}(\n-\p).$ Since for all large $\n,$ $f_{M^\prime}(\n)-\chi_{M^\prime}(\n)=0,$ we get the required result.
\eepf
\bp\label{result 2}
Let $S^\prime$ be a $\ZZ^s$-graded ring and 
$S=\bigoplus_{\n \in \NN^s} S^\prime_{\n}.$ 
Then 
$H^i_{S_{++}}(S^\prime) \cong H^i_{S_{++}}(S)$ for all $i >1 $ and we have the exact sequence 
$$0 \longrightarrow H^0_{S_{++}}(S) \longrightarrow H^0_{S_{++}}(S^\prime) \longrightarrow 
\frac{S^\prime}{S} \longrightarrow H^1_{S_{++}}(S) \longrightarrow H^1_{S_{++}}(S^\prime)
\longrightarrow 0 .$$

\ep
\begin{proof}
 Consider the short exact sequence of $S$-modules
 $$0 \longrightarrow S \longrightarrow S^\prime \longrightarrow \frac{S^\prime}{S} 
 \longrightarrow 0.$$
 This gives the long exact sequence of $S$-modules
 $$\cdots \longrightarrow H^i_{S_{++}}(S) \longrightarrow H^i_{S_{++}}(S^\prime) 
 \longrightarrow H^i_{S_{++}}\left(\frac{S^\prime}{S}\right) \longrightarrow \cdots .$$
 Since $\frac{S^\prime}{S}$ is $S_{++}$-torsion, $H^0_{S_{++}}\left(\frac{S^\prime}{S}\right)=
 \frac{S^\prime}{S}$ and $H^i_{S_{++}}\left(\frac{S^\prime}{S}\right)=0$ for all $i >0.$ 
 Hence the result follows.
\end{proof}

The GSF for multi-graded Rees algebras proved below  generalises the theorems \cite[Theorem 5.1]{jv}, \cite[Theorem 1]{mv}  and \cite[Theorem 4.1]{blancafort}.

\begin{theorem}\cite[Theorem 4.3]{msv}\label{r3}
Let $(R,\mm)$ be a Noetherian local ring of dimension $d$ and $\idl$ be $\mm$-primary ideals of $R.$ Let $\fil=\lbrace\fil(\n)\rbrace_{\n\in{\ZZ}^s}$ be an $\I$-admissible filtration of ideals in $R.$ Then
\ben
{\item[(1)] $h_{\R_{++}}^i(\mathcal{R}'(\fil))_{\n}<\infty$ for all $i\geq 0$ and $\n\in\ZZ^s.$ 
\item[(2)] $\po(\n)-\ho(\n)=\sum\limits_{i\geq 0}(-1)^i h_{\R_{++}}^i(\mathcal{R}'(\fil))_{\n}$ 
for all $\n\in\ZZ^s.$}\een
\end{theorem}
\bpf
 $(1)$ Denote $\displaystyle\frac{\mathcal{R}'(\fil)}{\mathcal{R}'(\fil)(e_i)}$ by $G'_i(\fil).$ By the change of ring principle, $H_{{G_i(\I)}_{++}}^j(G'_i(\fil))\cong H_{{\R}_{++}}^j(G'_i(\fil))$ for all $\ii$ and $j\geq 0.$ For a fixed $i,$ consider the short exact sequence of $\R(\I)$-modules 
 \begin{equation}\label{equation}0\longrightarrow \mathcal{R}'(\fil)(e_i)\longrightarrow \mathcal{R}'(\fil)\longrightarrow G'_i(\fil)\longrightarrow 0.
 \end{equation}
 This induces the long exact sequence of $R$-modules
$$0\longrightarrow [H_{{\R}_{++}}^0(\mathcal{R}'(\fil))]_{\n+e_i} \longrightarrow [H_{{\R}_{++}}^0(\mathcal{R}'(\fil))]_{\n} 
\longrightarrow [H_{{\R}_{++}}^0(G'_i(\fil))]_{\n}
 \longrightarrow [H_{{\R}_{++}}^1(\mathcal{R}'(\fil))]_{\n+e_i} \longrightarrow \cdots.$$ 
  By Propositions \ref{vl} and \ref{result 2}, $[H_{{\R}_{++}}^j(\mathcal{R}'(\fil))]_{\n}=0$ for all 
  large $\n$ and $j\geq 0.$ Since $\lf\frac{G'_i(\fil)}{\ggi}\rg_{\n}=0$ for all $\n \in \NN^s$ or  $n_i<0,$ by Propositions \ref{vl} and \ref{result 2}, 
  $[H_{{\R}_{++}}^j(G'_i(\fil))]_{\n}$ is finitely generated $(G_i(\I))_{\underline{0}}$-module 
 for all $\n \in \NN^s$ or  $n_i<0$ and $j\geq 0.$ Since $(G_i(\I))_{\underline{0}}$ is Artinian, 
  $[H_{{\R}_{++}}^j(G'_i(\fil))]_{\n}$ has finite length for all $\n \in \NN^s$ or  $n_i<0$ and $j\geq 0.$ Hence using decreasing induction on $\n,$ we get that 
  $h_{\R_{++}}^j(\mathcal{R}'(\fil))_{\n}<\infty$ for all $j\geq 0$ and $\n\in\ZZ^s.$ 
\\ $(2)$ Let $\chi_M(\n)=\sum\limits_{i\geq 0}(-1)^ih_{{\R}_{++}}^i(M)_{\n}$ where 
$M$ is an $\mathcal{R}(\I)$-module. Then from the short exact sequence (\ref{equation}), Theorem \ref{GS} and Proposition \ref{result 2}, for each $\ii$ and $\n\in\NN^s$ or $n_i<0,$
 \beqn
 \chi_{\mathcal{R}'(\fil)}(\n+e_i)-\chi_{\mathcal{R}'(\fil)}(\n)&=& -\chi_{G'_i(\fil)}(\n)\\&=&  -\chi_{G_i(\fil)}(\n)\\&=& P_{G_i(\fil)}(\n)-H_{G_i(\fil)}(\n)\\&=&(\po(\n+e_i)-\po(\n))-(\ho(\n+e_i)-\ho(\n)).
 \eeqn
Let $h(\n)=\chi_{\mathcal{R}'(\fil)}(\n)-(\po(\n)-\ho(\n)).$ Then $h(\n+e_i)=h(\n)$ for all $\n\in\NN^s$ or $n_i<0$ and $\ii.$ Since $h(\n)=0$ for all large $\n,$ $h(\n)=0$ $\fa$ $\n\in\ZZ^s.$
\eepf

\section{Formulas for local cohomology modules}

In this section we derive formulas for the graded components of
the local cohomology modules of certain Rees rings and associated graded rings in terms of the Ratliff-Rush closure filtration of a multi-graded filtration of ideals. These generalise \cite[Proposition 2.5 and Theorem 3.5]{blancafort}. We use these formulas  to derive various  properties of the Hilbert coefficients in further sections. 

In the following proposition we derive a formula for $H_{G(\fil)_+}^d(G(\fil))_n.$ 

\bp\label{highest local cohomology}
Let $(R,\mm)$ be a \CM local ring of dimension $d\geq 1,$ $I$ an $\mm$-primary ideal and $\fil=\{I_n\}$ be an $I$-admissible filtration of ideals in $R.$ Let $(x_1,\ldots, x_d)$ be a minimal reduction of $\fil.$ Put $\underline x^k=(x_1^k,\ldots, x_d^k)$ for all $k\geq 1.$ Then for all  $n\in\ZZ,$
 $$H_{G(\fil)_+}^d(G(\fil))_n=\displaystyle\lim_{\stackrel{\longrightarrow}{k}}\frac{I_{dk+n}}{\underline x^k I_{(d-1)k+n}+I_{dk+n+1}}.$$
\ep
\bpf
Let $x_i^*=x_i+I_2$ be the image of $x_i$ in $G(\fil).$ Since $\sqrt{G(\fil)_+}=\sqrt{(x_1^*,\ldots,x_d^*)},$ by \cite[Theorem 5.2.9]{BS}, $H_{G(\fil)_+}^d(G(\fil))=\displaystyle\lim_{\stackrel{\longrightarrow}{k}} H^d((x_1^*)^k,\ldots,(x_d^*)^k,G(\fil))$ where $H^d((x_1^*)^k,\ldots,(x_d^*)^k,G(\fil))$ is the $d$-th cohomology of the Koszul complex of $G(\fil)$ \wrt the elements $(x_1^*)^k,\ldots,(x_d^*)^k.$ Thus we get the required result.
\eepf
\bp\label{hlcr}
Let $(R,\mm)$ be a \CM local ring of dimension $d\geq 1,$ $I$ an $\mm$-primary ideal and $\fil=\{I_n\}$ be an $I$-admissible filtration of ideals in $R.$ Let $(x_1,\ldots, x_d)$ be a minimal reduction of $\fil.$ Put $\underline x^k=(x_1^k,\ldots, x_d^k)$ for all $k\geq 1.$ Then for all  $ n\in\ZZ, $
$$H_{\R(\fil)_+}^d(\R(\fil))_n=\displaystyle\lim_{\stackrel{\longrightarrow}{k}}\frac{I_{dk+n}}{\underline x^k I_{(d-1)k+n}}
    .$$
\ep
\bpf
Since $\sqrt{\R(\fil)_+}=\sqrt{(x_1t,\ldots,x_dt)},$ we have  $H_{\R(\fil)_+}^d(\R(\fil))=\displaystyle\lim_{\stackrel{\longrightarrow}{k}} H^d((x_1t)^k,\ldots,(x_dt)^k,\R(\fil))$ by \cite[Theorem 5.2.9]{BS} where $H^d((x_1t)^k,\ldots,(x_dt)^k,\R(\fil))$ is the $d$-th cohomology of the Koszul complex of $\R(\fil)$ \wrt the elements $(x_1t)^k,\ldots,(x_dt)^k.$ Thus we get the required result.
\eepf
\begin{lemma}{\em[Rees' Lemma]}\cite[Lemma 1.2]{rees3}\cite[Lemma 2.2]{msv} 
\label{one}
 Let $(R,\mm,k)$ be a Noetherian local ring of dimension $d$ with infinite residue field $k$ and 
 $\idl$ be $\mm$-primary ideals of $R.$ Let $\fil=\lbrace\fil(\n)\rbrace_{\n\in{\ZZ}^s}$ be an 
 $\I$-admissible filtration of ideals in $R$ and $S$ be a finite set of prime ideals of $R$ not 
 containing $I_1\cdots I_s.$ Then for each $\ii,$ there exists an element $x_i\in I_i$ not contained 
 in any of the prime ideals of $S$ and an integer $r_i$ such that for all $\n \geq r_ie_i,$ 
 $$\fil(\n)\cap (x_i)=x_i\fil(\n-e_i).$$ 
\el
\bt\cite[Theorem 1.3]{rees3}\cite[Theorem 2.3]{msv}\label{eleven}
Let $(R,\mm)$ be a Noetherian local ring of dimension $d$ with infinite residue field and $\idl$ be  $\mm$-primary ideals of $R.$ Let $\fil=\lbrace\fil(\n)\rbrace_{\n\in \ZZ^s}$ be an $\I$-admissible filtration of ideals in $R.$ Then there exist a set of elements ${\lbrace}{x_{ij}\in{I_i}: j=1,\ldots,d;\ii}\rbrace$ such that $(y_1,\ldots,y_d)\fil(\n)=\fil(\n+e)$ for all large $\n$ where $y_j=x_{1j}\cdots x_{sj}\in\id$ for all $j=1,\ldots,d.$
Moreover, if the ring is Cohen-Macaulay local then there exist elements $x_{i1}\in{I_i}$ and integers $r_i$ for all $\ii$ such that for all $\n \geq r_ie_i,$ 
 $\fil(\n)\cap (x_{i1})=x_{i1}\fil(\n-e_i)$ and $y_1=x_{11}\cdots x_{s1}.$ 

\et
\bp\cite[Proposition 3.5]{msv}\label{zero}
Let $(R,\mm)$ be a Cohen-Macaulay local ring of dimension two with infinite residue field and $\idl$ be $\mm$-primary ideals in $R.$ Let  
$\fil=\lbrace\fil(\n)\rbrace_{\n\in{\ZZ}^s}$ be an $\I$-admissible filtration of ideals in $R.$  
Then for all $\n\in\NN^s,$ 
$$[H_{\R(\fil)_{++}}^1(\R(\fil))]_{\n}\cong\frac{\breve\fil(\n)}{\fil(\n)}.$$
\bpf
By Lemma \ref{one} and Theorem \ref{eleven}, there exists a regular sequence $\lbrace y_1,y_2\rbrace$ such that $(y_1,y_2)\fil(\n)=\fil(\n+e)$ for all large $\n.$ For all $k\geq 1,$ consider the following complex of $\R(\fil)$-modules
$$F^{k.} : 0\longrightarrow \R(\fil)\overset{\alpha_k}\longrightarrow \R(\fil)(ke)^2\overset{\beta_k}
\longrightarrow \R(\fil)(2ke)\longrightarrow 0,$$ where 
$\alpha_k(1)=({y_1}^k{\underline{t}}^{ke},{y_2}^k{\underline{t}}^{ke})$
 and $\beta_k(u,v)={y_2}^k{\underline{t}}^{ke}u-{y_1}^k{\underline{t}}^{ke}v.$ Since radical of the ideal $(y_1\underline t^e,y_2\underline t^e)\R(\fil)$ is same as radical of the ideal $\R(\fil)_{++},$ by \cite[Theorem 5.2.9]{BS},
\beqn \label{second local cohomology}
[H^1_{\R(\fil)_{++}}(\mathcal{R}(\fil))]_{\n} \cong \displaystyle\lim_{\stackrel{\longrightarrow}{k}}
\frac{(\ker{\beta_k)_{\n}}}{(\im{\alpha_k})_{\n}}.
\eeqn Suppose $(u,v)\in (\ker\beta_k)_{\n}$ for any $\n\in\NN^s.$ Then 
 ${y_2}^ku-{y_1}^kv=0.$ Since $\lbrace y_1,y_2\rbrace$ is a regular sequence, 
 $u={y_1}^kp$ for some $p\in R.$ Thus $v={y_2}^kp.$ Hence $(u,v)=({y_1}^kp,{y_2}^kp).$ This implies for all $\n\in\NN^s,$
$(u,v)\in (\ker\beta_k)_{\n}$ if and only if $(u,v)=({y_1}^kp,{y_2}^kp)$ for 
 some $p\in (\fil(\n+ke):({y_1}^k,{y_2}^k)).$ For $k\gg0,$ by \cite[Proposition 3.1]{msv}, 
 $\breve\fil(\n)=(\fil(\n+ke):({y_1}^k,{y_2}^k))$ for all $\n\in\NN^s.$ Hence for all $\n\in\NN^s$ and $k\gg 0,$ $(\ker\beta_k)_{\n}\cong\breve\fil(\n).$ 
 Also for all $\n\in\NN^s,$ $$(\im\alpha_k)_{\n}=\lbrace ({y_1}^kp{\underline{t}}^{ke},{y_2}^kp{\underline{t}}^{ke}):p\in\R(\fil)_{\n}\rbrace\cong\fil(\n).$$ 
 Hence $\displaystyle [H_{\R(\fil)_{++}}^1(\R(\fil))]_{\n}\cong\frac{\breve\fil(\n)}{\fil(\n)}$ for all $\n\in\NN^s.$ 
 \eepf 
\ep
\bp\label{h222}
Let $(R,\mm)$ be a Cohen-Macaulay local ring of dimension two with infinite residue field, $I$ an $\mm$-primary ideal and $\fil=\lbrace I_n\rbrace_{n\in{\ZZ}}$ be an $I$-admissible filtration of ideals in $R.$ 
Then \[ [H_{\R(\fil)_{+}}^1(\R(\fil))]_{n} = \left\{
  \begin{array}{l l}
   \breve{I_n}/I_n  & \quad \text{if $n\geq 0$ }\\
    R & \quad \text{if $n<0.$ }
  \end{array} \right. \] 
\ep
\bpf
By Proposition \ref{zero}, we get $[H_{\R(\fil)_{+}}^1(\R(\fil))]_{n} = \breve{I_n}/I_n$ for all $n\geq 0.$ Let $J$ be minimal reduction of $\fil$ generated by superficial sequence $y_1,y_2.$
For all $k\geq 1,$ consider the following complex of $\R(\fil)$-modules
$$F^{k.} : 0\longrightarrow \R(\fil)\overset{\alpha_k}\longrightarrow \R(\fil)(k)^2\overset{\beta_k}
\longrightarrow \R(\fil)(2k)\longrightarrow 0,$$ where 
$\alpha_k(1)=({y_1}^k{{t}}^{k},{y_2}^k{{t}}^{k})$
 and $\beta_k(u,v)={y_2}^k{{t}}^{k}u-{y_1}^k{{t}}^{k}v.$ Since radical of the ideal $(y_1t,y_2t)\R(\fil)$ is same as radical of the ideal $\R(\fil)_{+},$ by \cite[Theorem 5.2.9]{BS},
\beqn \label{second local cohomology}
[H^1_{\R(\fil)_{+}}(\mathcal{R}(\fil))]_{n} \cong \displaystyle\lim_{\stackrel{\longrightarrow}{k}}
\frac{(\ker{\beta_k)_{n}}}{(\im{\alpha_k})_{n}}.\eeqn Now for $n<0,$ $\R(\fil)_n=0.$ Hence $(\im{\alpha_k})_{n}=0.$
\\Suppose $(u,v)\in (\ker\beta_k)_{n}$ for any $n<0.$ Then 
 ${y_2}^ku-{y_1}^kv=0.$ Since $\lbrace y_1,y_2\rbrace$ is a regular sequence, 
 $u={y_1}^kp$ for some $p\in R.$ Thus $v={y_2}^kp.$ Hence $(u,v)=({y_1}^kp,{y_2}^kp).$ This implies for all $n<0,$
$(u,v)\in (\ker\beta_k)_{n}$ if and only if $(u,v)=({y_1}^kp,{y_2}^kp)$ for 
 some $p\in (\fil(n+k):({y_1}^k,{y_2}^k))=R.$
\eepf

\bp\label{h1}\cite[Theorem 3.5]{blancafort}
Let $(R,\mm)$ be a Cohen-Macaulay local ring of dimension two with infinite residue field, $I$ an $\mm$-primary ideal and $\fil=\lbrace I_n\rbrace_{n\in{\ZZ}}$ be an $I$-admissible filtration of ideals in $R.$ 
Then \[ [H_{\R(\fil)_{+}}^1(\R^\prime(\fil))]_{n} = \left\{
  \begin{array}{l l}
   \breve{I_n}/I_n  & \quad \text{if $n\geq 0$ }\\
    0 & \quad \text{if $n<0.$ }
  \end{array} \right. \] 
\ep
\bpf
Since $ [H_{\R(\fil)_{++}}^1(\R^\prime(\fil))]_{n}= [H_{\R(\fil)_{++}}^1(\R(\fil))]_{n}$ for all $n\in\NN$ by Proposition \ref{result 2}, using Proposition \ref{h222}, we get $[H_{\R(\fil)_{++}}^1(\R^\prime(\fil))]_{n}=\breve{I_n}/I_n.$
\\Let $J$ be minimal reduction of $\fil$ generated by superficial sequence $y_1,y_2.$
For all $k\geq 1,$ consider the following complex of $\R(\fil)$-modules
$$F^{k.} : 0\longrightarrow \R^\prime(\fil)\overset{\alpha_k}\longrightarrow \R^\prime(\fil)(k)^2\overset{\beta_k}
\longrightarrow \R^\prime(\fil)(2k)\longrightarrow 0,$$ where 
$\alpha_k(1)=({y_1}^k{{t}}^{k},{y_2}^k{{t}}^{k})$
 and $\beta_k(u,v)={y_2}^k{{t}}^{k}u-{y_1}^k{{t}}^{k}v.$ Since radical of the ideal $(y_1\underline t,y_2\underline t)\R(\fil)$ is same as radical of the ideal $\R(\fil)_{+},$ by \cite[Theorem 5.2.9]{BS},
\beqn \label{second local cohomology}
[H^1_{\R(\fil)_{+}}(\mathcal{R}^\prime(\fil))]_{n} \cong \displaystyle\lim_{\stackrel{\longrightarrow}{k}}
\frac{(\ker{\beta_k)_{n}}}{(\im{\alpha_k})_{n}}.\eeqn
for all $n\in\ZZ\setminus\NN,$ $$(\im\alpha_k)_{n}=\lbrace ({y_1}^kp{\underline{t}}^{ke},{y_2}^kp{\underline{t}}^{ke}):p\in\R^\prime(\fil)_{n}=R\rbrace\cong R.$$ Thus $[H_{\R(\fil)_{+}}^1(\R^\prime(\fil))]_{n}=0$ for all $n\in\ZZ\setminus\NN.$
\eepf

\bl\cite[Lemma 2.11]{msv}\label{result 4}  
Let $\idl$ be $\mm$-primary ideals in a Noetherian local ring $(R,\mm)$ of dimension $d\geq 1$ such that grade$(\id)\geq 1.$ 
Let $\fil=\lbrace\fil(\n)\rbrace_{\n\in \ZZ^s}$ be an $\I$-admissible filtration of ideals in $R.$ Denote $\R(\I)_{++}$ as $\R_{++}.$ Then 
$$\lm_R[H_{\R_{++}}^d(\R'(\fil))]_{\n}\leq\lm_R[H_{\R_{++}}^d(\R'(\fil))]_{\n-e_i}$$ 
for all $\n\in\ZZ^s$ and  $\ii.$
\bpf
By Lemma \ref{one} and Theorem \ref{eleven}, there exists an ideal $J=(\yy)\subseteq \id$ such that $y_1=x_{11}\cdots x_{s1}$ is a
nonzerodivisor, $x_{i1}\in I_i$ $\fa$ $\ii$ and $J\fil(\n)=\fil(\n+e)$ for all large $\n.$ Hence 
$\sqrt{\R(\I)_{++}}=\sqrt{(y_1\underline t,\ldots,y_d\underline t)}.$ Consider the short exact sequence of $\R(\I)$-modules,
$$0\longrightarrow \R'(\fil)(-e_i)\overset{x_{i1}t_i}\longrightarrow\R'(\fil)\longrightarrow 
\frac{\R'(\fil)}{x_{i1}t_i{\R'(\fil)}}\longrightarrow 0.
$$ This gives a long exact sequence of $\n$-graded components of local cohomology modules, \beqn 
 \cdots \longrightarrow [H_{\R(\I)_{++}}^d(\R'(\fil))]_{\n-e_i}\longrightarrow 
 [H_{\R(\I)_{++}}^d(\R'(\fil))]_{\n}\longrightarrow \left[H_{\R(\I)_{++}}^d\left(\frac{\R'(\fil)}{x_{i1}t_i{\R'(\fil)}}\right)\right]_{\n}
 \longrightarrow 0.
 \eeqn
Let $T=\frac{\R(\I)}{x_{i1}t_i\R(\I)}.$ Now $\frac{\R'(\fil)}{x_{i1}t_i{\R'(\fil)}}$ is a $T$-module and $\sqrt{\left(\frac{\R(\I)}{x_{i1}t_i\R(\I)}\right)_{++}}=\sqrt{({y_2}\underline t,\cdots,{y_d}\underline t)T}.$ Hence $H_{\R(\I)_{++}}^d\lf\frac{\R'(\fil)}{x_{i1}t_i{\R'(\fil)}}\rg=0$ which implies the required result.
\eepf
\el
\begin{theorem}\cite[Theorem 3.3]{msv}\label{two}
Let $(R,\mm)$ be a Noetherian local ring of dimension $d\geq 1$ with infinite residue field and $\idl$ be $\mm$-primary ideals in $R$ such that grade$(\id)\geq 1.$ Let $\fil=\lbrace\fil(\n)\rbrace_{\n\in \ZZ^s}$ be an $\I$-
admissible filtration of ideals in $R.$ Then for all $\n\in \NN^s$ and $\ii,$ 
$$[H^0_{G_i(\mathcal F)_{++}}(G_i(\mathcal F))]_{\n}=\frac{\breve\fil(\n+e_i)\cap\fil(\n)}{\fil(\n+e_i)}.
$$ \et
\bpf
Let $x\in\fil(\n)$ and $x^{*}=x+\fil(\n+{e_i})\in [H^0_{G_i(\mathcal F)_{++}}(G_i(\mathcal F))]_{\n}.$ 
Then $x^{*}\gi^{k}=0$ for some $k\geq 1.$ Therefore $x\fil({ e})^k\subseteq\fil(\n+k{ e}+{ e_i}).$ 
Hence $x\in\breve\fil(\n+{ e_i}).$ 
\\ Conversely, suppose $x^*\in {\breve\fil(\n+{ e_i})\cap\fil(\n)}/{\fil(\n+{e_i})}.$ We show that there exists $m\gg 0$ such that $x^*{\gi^{m}}=0.$ Since $\displaystyle\gi^{m}\subseteq\bigoplus\limits_{{\underline p}\geq m{ e}}\fil({\underline p})/\fil({\underline p+e_i}),$ it is enough to show that $x^*(\fil({\underline p})/\fil({\underline p+e_i}))=0$ for all large ${\underline p}.$ By \cite[Proposition 3.1]{msv}, there exists $\m\in\NN^s$ with $\m\geq{ e}$ such that $\breve\fil(\rrr)=\fil(\rrr)$ for all $\rrr\geq \m.$ Thus for all $\rrr\geq \m,$ $$x\fil(\rrr)\subseteq \breve\fil(\n+{e_i})\fil(\rrr)\subseteq\breve\fil(\n+\rrr+{e_i})=\fil(\n+\rrr+{e_i}).$$
Therefore $(x+\fil(\n+{e_i}))G_i(\fil)_{++}^m=0$ 
for some $m \geq 1.$ Hence $(x+\fil(\n+{e_i}))\in[H^0_{G_i(\mathcal F)_{++}}(G_i(\mathcal F))]_{\n}.$
\eepf
\section{The postulation number and the reduction number}
In \cite[Proposition 3]{sa} Sally gave a nice relation between the postulation number and the reduction number of the filtration $\{\mm^n\}_{n \in \NN}.$ In \cite[Corollary 3.8]{marleythesis} Marley generalised this  relation for any $I$-admissible filtration. In this section we derive these results using the Grothendieck-Serre formula. We recall few preliminary results about superficial sequences which are useful to apply induction in the study of Hilbert coefficients.  
\\Let $(R,\mm)$ be a Noetherian local ring, $I$ be an $\mm$-primary ideal and $\fil=\{I_n\}_{n\in\ZZ}$ be an $I$-admissible filtration of ideals in $R.$

 \bd An element $x\in I_t\setminus I_{t+1}$ is called {\bf{superficial element for $\fil$ of degree $t$}} if there exists an integer $c\geq 0$ such that $(I_{n+t}:x)\cap I_c=I_n\mbox{ for all }n\geq c.$\ed
 If the residue field of $R$ is infinite then there exists a superficial element of degree $1$ \cite[Proposition 2.3]{rho}. If $\grade(I_1)\geq 1$ and $x\in I_1$ is superficial for $\fil,$ Huckaba and Marley \cite{huckaba-marley}, showed that $x$ is nonzerodivisor in $R$ and $(I_{n+1}:x)=I_n$ for all large $n.$  If dimension of $R$ is $d\geq 1,$ $x\in I_1\setminus I_2$ is superficial element for $\fil$ and $x$ is a nonzerodivisor on $R$ then by \cite[Lemma A.2.1]{marleythesis}, $e_i(\fil)=e_i(\fil^\prime)$ for all $0\leq i<d$ where $R^\prime=R/(x)$ and $\fil^\prime=\{I_nR^\prime\}_{n\in\ZZ}.$ The following lemma is due to Blancafort \cite[Lemma 3.1.6]{blancafort thesis}. This lemma was first proved by Huckaba \cite[Lemma 1.1]{huc} for $I$-adic filtration.
\begin{lemma}\label{redn1}
Let $(R,\mm)$ be a \CM local ring of dimension $d\geq 1,$ $I$ an $\mm$-primary ideal and $\fil=\{I_n\}_{n\in \ZZ}$ be an $I$-admissible filtration of ideals in $R. $ Suppose $J$ is a minimal reduction of $\fil$ and there exists an $x\in J\setminus I_2$ such that $x^*=x+I_2$ is a nonzerodivisor in $G(\fil).$ Let $R^\prime=R/(x).$ Then $r(\fil)=r(\fil^\prime)$ where $\fil^\prime=\{I_n R^\prime\}_{n\in\ZZ}.$
\end{lemma}
\bpf
We denote $r(\fil)$ and $r(\fil^\prime)$ by $r$ and $s$ respectively. It is clear that $ s\leq r.$ We use the notation `` $'$ " to denote the image in $R^\prime.$  Let $n \geq s$ and $a\in I_{n+1}.$ Then $a^\prime\in J'I_{n}'.$ Hence $a=p+xq$ for some $p\in JI_n$ and $q\in R.$ Therefore $xq\in I_{n+1}$ which implies $q \in (I_{n+1}:x).$ Since $x^*$ is a nonzerodivisor in $G(\fil),$ we have $(I_{n+1}:x)=I_n$ for all $n\in\ZZ.$ Hence we get the required result.
\eepf
\bd
 If $\underline x=x_1,\dots,x_r\in I_1,$ we say $\underline x$ is a superficial sequence for $\fil$ if for all $0\leq i< r,$ $x_{i+1}$ is superficial for $\fil/(x_1,\dots,x_i).$\ed
 Suppose $(R,\mm)$ is Cohen-Macaulay local ring of dimension $d,$ $I_1$ is an $\mm$-primary ideal and $\fil=\{I_n\}$ is an $I_1$-admissible filtration of ideals in $R.$ Suppose $x_1,\dots,x_r\in I_1$ and $1\leq r\leq d,$ then $x_1,\dots,x_r$ is a superficial sequence for $\fil$ if and only if $x_1,\dots,x_r$ is $R$-regular sequence and there exists an integer $n_0\geq 0$ such that for all $1\leq i\leq r,$ $$(x_1,\dots,x_i)\cap I_n=(x_1,\dots,x_i)I_{n-1}\mbox{ for all }n\geq n_0.$$ This result was first proved by Valabrega and Valla \cite[Corollary 2.7]{vv} for $I$-adic filtration and then by Huckaba and Marley \cite{huckaba-marley} for $\ZZ$-graded admissible filtartions. Marley \cite[Proposition A.2.4]{marleythesis} showed that if residue the field is infinite then any minimal reduction of $\fil $ can be generated by a superficial sequence for $\fil.$ The following lemma is due to Huckaba and Marley \cite[Lemma 2.1]{huckaba-marley}.
\bl\label{hhh}\cite[Lemma 2.1]{huckaba-marley}
Let $(R,\mm)$ be a Noetherian local ring of dimension $d\geq 1,$ $I$ an $\mm$-primary ideal and $\fil=\{I_n\}$ be an $I$-admissible filtration of ideals in $R.$ Let $x_1,\ldots,x_k$ be a superficial sequence for $\fil.$ If $\grade G(\fil)_+\geq k$ then $x_1^*,\ldots,x_k^*$ is a regular sequence in $G(\fil)$ and hence $G(\fil)/(x_1^*,\dots,x_k^*)\simeq G(\fil/(x_1,\dots,x_k))$ where $x_i^*$ is image of $x_i$ in $G(\fil).$
\el
\bpf
By induction it is enough to prove for $k=1.$ Let $(I_{n+1}:x_1)\cap I_c=I_n\mbox{ for all }n\geq c.$ Let $x^*\in {(0:x_1^*)\cap G(\fil)_n}$ for some $n\in\NN.$ We show that $x^*(G(\fil)_{+})^{c+1}=0.$ Let $0\neq z^*\in {G(\fil)_{+}^{c+1}\cap G(\fil)_{ p}}.$ Now $x^*z^*\in G(\fil)_{n+p}$ and $x_1xz\in I_{n+p+2}.$ Therefore $xz\in (I_{n+p+2}:x_1)\cap I_c =I_{n+p+1}.$
Thus $x^*z^*=0$ in $G(\fil).$ Hence $x^*\in (0:_{G(\fil)}{(G(\fil)_+)^{c+1}})=0.$
\eepf
The next theorem was proved for the $\mm$-adic by  Sally \cite[Proposition 3]{sa}. We have adapted her proof for any admissible filtration. 
\bt\label{result sa ma}
Let $(R,\mm)$ be a \CM local ring of dimension $d\geq 1$ with infinite residue field, $I$ an $\mm$-primary ideal and $\fil=\{I_n\}$ be an $I$-admissible filtration of ideals in $R.$ Let $H_R(n)=\lm\lf{I_n}/{I_{n+1}}\rg$ and $P_R(X)\in \QQ[X]$ such that $P_R(n)=H_R(n)$ for all large $n.$ Suppose $\grade G(\fil)_+\geq d-1.$ Then for a minimal reduction $J=(x_1,\ldots, x_d)$ of $\fil,$ $H_R(r_J(\fil)-d)\neq P_R(r_J(\fil)-d)$ and $H_R(n)= P_R(n)$ for all $n\geq r_J(\fil)-d+1.$ 
\et
\bpf
 We denote $r_J(\fil)$ by $r.$ We use induction on $d.$ Let $d=1.$ Without loss of generality we assume $x_1$ is superficial.  Then 
 $$H_{G(\fil)_+}^0(G(\fil))_n=\{z^*\in {I_n}/{I_{n+1}} \mid zI_l\in I_{n+l+1}\mbox{ for all large }l \}.$$ 
 For $n\geq r-1,$ $zx_1^l\in I_{n+l+1}=x_1^lI_{n+1}$ implies $z\in I_{n+1}.$ Thus for all $n\geq r-1,$ $H_{G(\fil)_+}^0(G(\fil))_n=0.$ 
\\ Now we prove that $H_{G(\fil)_+}^1(G(\fil))_{r-1}\neq 0$ and $H_{G(\fil)_+}^1(G(\fil))_{n}=0$ for all $n\geq r.$ 
For each $n,$ consider the following map  
$$\frac{I_{k+n}}{x_1^k I_{n}+I_{k+n+1}}\overset{\phi_k}\longrightarrow \frac{I_{k+n+1}}{x_1^{k+1} I_{n}+I_{k+n+2}}\mbox{ where }\phi_k(\ov z)=\ov {x_1z}.$$ For all large $k,$ $I_{k+n+1}=x_1I_{k+n}.$ Hence for all large $k,$ $\phi_k$ is surjective. Now suppose $\phi_k(\ov z)=0$ for some $\ov z\in {I_{k+n}}/{x_1^k I_{n}+I_{k+n+1}}.$ Then $x_1z\in x_1^{k+1} I_{n}+I_{k+n+2}.$ Therefore $x_1z=x_1^{k+1}a+b$ for some $a\in I_n$ and $b\in I_{k+n+2}.$ Thus $b\in (x_1)\cap I_{k+n+2}.$ Since $x_1$ is superficial, for all large $k,$ $b\in x_1I_{k+n+1}$ and hence $z\in x_1^k I_{n}+I_{n+k+1}.$ Thus for all large $k,$ $\phi_k$ is injective. Therefore by Proposition \ref{highest local cohomology}, for all large $k,$ $$H_{G(\fil)_+}^1(G(\fil))_n\simeq \frac{I_{k+n}}{x_1^k I_{n}+I_{k+n+1}}.$$
 Thus for all $n\geq r$ and for all large $k,$ $I_{k+n}=x_1^k I_n\subseteq x_1^k I_n+I_{k+n+1}.$ Hence $H_{G(\fil)_+}^1(G(\fil))_n=0$ for all $n\geq r.$ 
 \\Suppose $H_{G(\fil)_+}^1(G(\fil))_{r-1}=0.$ Then for all large $k,$ $$I_{k+r-1}=x_1^k I_{r-1}+I_{k+r}\subseteq x_1^kI_{r-1}.$$ Let $a\in I_{k+r-2}.$ Then $x_1a\in I_{k+r-1}\subseteq x_1^kI_{r-1}$ implies $a\in x_1^{k-1} I_{r-1}.$ Thus $I_{k+r-2}=x_1^{k-1} I_{r-1}.$ Using this procedure repeatedly, we get $I_r=x_1I_{r-1}$ which is a contradiction. Thus $H_{G(\fil)_+}^1(G(\fil))_{r-1}\neq 0.$  Therefore by Theorem \ref{GS}, we get the required result.
\\ Suppose $d\geq 2.$ Without loss of generality we assume $x_1,\ldots, x_d$ is superficial sequence for $\fil.$ Since $\grade  G(\fil)_+\geq d-1,$ by Lemma \ref{hhh}, we have $x_1^*$ is a nonzerodivisor of $G(\fil).$ By \cite[Proposition 3.1.3]{blancafort thesis} $G(\fil)/(x_1^*)\simeq G(\fil/(x_1)).$  For all $n\in\ZZ,$ consider the following exact sequence \beqnn\label{for} 0\longrightarrow \frac{I_{n-1}}{I_{n}}\overset{x_1^*}\longrightarrow \frac{I_n}{I_{n+1}}\longrightarrow \frac{I_n}{x_1I_{n-1}+I_{n+1}}\simeq\frac{I_n+(x_1)}{I_{n+1}+(x_1)}\longrightarrow 0.\eeqnn Then for all $n\in\ZZ,$ $$H_{R/(x_1)}(n)=H_{R}(n)-H_{R}(n-1)\mbox{ and hence }P_{R/(x_1)}(n)=P_{R}(n)-P_{R}(n-1).$$ Since $\dim {R}/{(x_1)}=d-1$ and $\grade  {G(\fil/(x_1))}_+\geq d-2,$ by induction and Lemma \ref{redn1}, we have $$H_{R/(x_1)}(r-d+1)\neq P_{R/(x_1)}(r-d+1)\mbox{ and }H_{R/(x_1)}(n)= P_{R/(x_1)}(n)\mbox{ for all }n\geq r-d+2.$$ Since there exists an integer $m,$ such that for all $n\geq m,$ $P_{R}(n)=H_{R}(n),$ we have $$P_{R}(n-1)-H_{R}(
n-1)=P_{R}(n)-H_{R}(n)=\cdots=P_{R}(n+m)-H_{R}(n+m)=0\mbox{ for all }n\geq r-d+2.$$ Therefore \beqn 0&\neq& P_{R/(x_1)}(r-d+1)-H_{R/(x_1)}(r-d+1)\\&=&[P_{R}(r-d+1)-H_{R}(r-d+1)]-[P_{R}(r-d)-H_{R}(r-d)]=P_{R}(r-d)-H_{R}(r-d).\eeqn
 \eepf
The following result is due to Marley \cite[Corollary 3.8]{marleythesis}. Here we give another proof which follows from Theorem \ref{result sa ma}.
\begin{theorem}\cite[Corollary 3.8]{marleythesis}
\label{relating pn and rn}
Let $(R,\mm)$ be a \CM local ring of dimension $d\geq 1$ with infinite residue field, $I$ an $\mm$-primary ideal and $\fil=\{I_n\}$ be an $I$-admissible filtration of ideals in $R.$ Let $\grade G(\fil)_+\geq d-1.$ Then $r_J(\fil)=n(\fil)+d$ for any minimal reduction $J$ of $\fil.$ In particular, $r(\fil)=n(\fil)+d.$
\end{theorem}
\bpf
Let $H_R(n)=\lm\lf{I_n}/{I_{n+1}}\rg$ for all $n$ and $P_R(X)\in \QQ[X]$ such that $P_R(n)=H_R(n)$ for all large $n.$ Let $d=1$ and $J$ be any minimal reduction of $\fil.$ Denote $r_J(\fil)$ by $r.$ Then degree of the polynomial $P_R(X)$ is zero. Hence $P_R(X)=a$ where $a$ is a nonzero constant. By Theorem \ref{result sa ma}, for all $n\geq r,$ $P_R(n)=H_R(n).$ Therefore for all $n\geq r,$ we have $$\lm\lf\frac{R}{I_{n}}\rg=(n-r)a+\lm\lf\frac{R}{I_{r}}\rg=na+(\lm\lf\frac{R}{I_{r}}\rg-ra).$$ Hence $\po(n)=\ho(n)$ for all $n\geq r.$ Suppose $\po(r-1)=\ho(r-1).$ Then $$-a+\lm\lf\frac{R}{I_{r}}\rg=\lm\lf\frac{R}{I_{r-1}}\rg.$$ This implies $P_R(r-1)= H_R(r-1)$ which contradicts Theorem \ref{result sa ma}. Thus $r_J(\fil)-1=n(\fil)$ for any minimal reduction $J$ of $\fil.$ Hence we get the result for $d=1.$
\\Suppose $d\geq 2$ and $J=(x_1,\ldots, x_d)$ is a minimal reduction of $\fil$ consisting of superficial elements. Denote $r_J(\fil)$ by $r.$ For all $n\in\ZZ,$ we get $$H_{R}(n)=\ho(n+1)-\ho(n)\mbox{ and hence }P_{R}(n)=\po(n+1)-\po(n).$$ Since there exists an integer $m,$ such that for all $n\geq m,$ $\po(n)=\ho(n),$ by Theorem \ref{result sa ma}, we have $$\po(n)-\ho(n)=\po(n+1)-\ho(n+1)=\cdots=\po(n+m)-\ho(n+m)=0\mbox{ for all }n\geq r-d+1.$$ Again using Theorem \ref{result sa ma}, we get \beqn 0&\neq& P_{R}(r-d)-H_{R}(r-d)\\&=&[\po(r-d+1)-\ho(r-d+1)]-[\po(r-d)-\ho(r-d)]=\po(r-d)-\ho(r-d).\eeqn Thus $r_J(\fil)-d=n(\fil)$ for any minimal reduction $J$ of $\fil.$ Hence $r(\fil)=n(\fil)+d.$
\eepf

\section{Nonnegativity and vanishing of Hilbert coefficients}
In this section we apply Grothendieck-Serre formula to derive various properties of the Hilbert coefficients. We derive a result of Northcott, Narita, Marley and Itoh. 
We also derive a formula for the components of local cohomology modules of Rees algebras in terms of the Hilbert coefficients (Proposition \ref{formula for lc}) which generalises \cite[Proposition 5]{sa} and \cite[Proposition 3.3]{1jv}. 

The following theorem is a generalisation of a result due to Northcott \cite[Theorem 1]{northcott}.
\bt{\em{(Northcott's inequality)}}\label{n1}
Let $(R,\mm)$ be a $d\geq 1$-dimensional \CM local ring, $I$ an $\mm$-primary ideal and $\fil=\lbrace I_n\rbrace_{n\in{\ZZ}}$ be an 
 $I$-admissible filtration of ideals in $R.$ Then $$e_1{(\fil)}\geq e_0{(\fil)}-\lm\lf\frac{R}{I_1}\rg\geq 0.$$
\et
\bpf
 We use induction on $d.$ Let $d=1.$ Since $R$ is Cohen-Macaulay, putting $n=1$ in the Difference Formula (Theorem \ref{r3}) for Rees algebra of $\fil$, we have \beqn e_0{(\fil)}-e_1{(\fil)}-\lm\lf\frac{R}{I_1}\rg&=&\po(1)-\ho(1)\\&=&\lm_{R}[H_{\R(\fil)_{+}}^0(\mathcal{R}'(\fil))]_{1}-\lm_{R}[H_{\R(\fil)_{+}}^1(\mathcal{R}'(\fil))]_{1}\\&=&-\lm_{R}[H_{\R(\fil)_{+}}^1(\mathcal{R}'(\fil))]_{1}\leq 0.\eeqn Thus we get the first inequality. Suppose $d\geq 2$ and the result is true for rings with dimension upto $d-1.$ \wlg we may assume that the residue field of $R$ is infinite. Let $x\in I_1$ be a superficial element for $\fil.$ Then $e_0(\fil)=e_0(\fil')$ and $e_1(\fil)=e_1(\fil')$ where `` $'$ " denotes the image in $R'=R/(x).$ Since $\lm(R'/I_1')=\lm(R/I_1),$ by induction hypothesis we get the first inequality.
 \\ For any minimal reduction $J$ of $\fil,$ $J$ is minimal reduction  $I_1$ by \cite[Lemma 1.5]{rees3}. Hence  we get the second inequality.
\eepf
\bt\cite[Theorem 5.6]{msv}\label{all}
Let $(R,\mm)$ be a Cohen-Macaulay local ring of dimension $d\geq 1$ and $\idl$ be $\mm$-primary ideals of $R.$ Let $\fil=\lbrace\fil(\n)\rbrace_{\n\in\ZZ^s}$ be an $\I$-admissible filtration. Then 
\ben{
\item[(1)] $e_{\alpha}(\fil)\geq 0$ where $\alpha=({\alpha}_1,\ldots,{\alpha}_s)\in\NN^s,$ 
$|\alpha|\geq d-1.$
\item[(2)] $e_{\alpha}(\fil)\geq 0$ where $\alpha=({\alpha}_1,\ldots,{\alpha}_s)\in\NN^s,$ $|\alpha|= d-2$ and $d\geq 2.$}\een
\bpf
$(1)$ For $|\alpha|=d,$ the result follows from \cite[Theorem 2.4]{rees3}. Suppose $|\alpha|=d-1.$ 
We use induction on $d.$ Let $d=1.$ Then putting $\n=\underline 0$ in the Difference Formula 
(Theorem \ref{r3}), we get $e_{\underline 0}(\fil)=\lm_{R}[H_{\R_{++}}^1(\mathcal{R}'(\fil))]_{\underline 0}\geq 0.$ 
Let $d\geq 2$ and assume the result  for rings of dimension $d-1.$ Then there exists $i$ such 
that $\alpha_i\geq 1.$ \wlg assume $\alpha_1\geq 1.$ By Lemma \ref{one}, there exists a nonzerodivisor 
$x\in I_1$ such that $(x)\cap\fil(\n)=x\fil(\n-e_1)$ for all $\n\in\NN^s$ such that $
n_1\gg0.$ Let $R^\prime={R}/{(x)}$ and 
$\fil^\prime=\lbrace\fil(\n)R^\prime\rbrace_{\n\in\ZZ^s}.$ For 
all large $\n,$ consider the following short exact sequence 
$$0\longrightarrow\frac{(\fil(\n):(x))}{\fil(\n-e_1)}\longrightarrow\frac{R}{\fil(\n-e_1)}\overset{.x}
\longrightarrow\frac{R}{\fil(\n)}\longrightarrow\frac{R}{(x,\fil(\n))}\longrightarrow 0.$$ 
Since $\fa$ large $\n,$ $(\fil(\n):(x))=\fil(\n-e_1),$ we get $P_{\fil'}(\n)=\po(\n)-\po(\n-e_1).$ 
Hence $(-1)^{d-1-|(\alpha-e_1)|}b_{(\alpha-e_1)}(\fil')=(-1)^{d-|\alpha|}e_{\alpha}(\fil)$ where 
$$
P_{\fil'}(\n)=\displaystyle\sum\limits_{\substack{\gamma=({\gamma}_1,\ldots,{\gamma}_s)\in{\NN}^s \\ |\gamma|\leq d-1}}(-1)^{d-1-|{\gamma}|}{b_{\gamma}}(\fil')\binom{{n_1}+{{\gamma}_1}-1}{{\gamma}_1}\cdots\binom{{n_s}+{{\gamma}_s}-1}{{\gamma}_s}.$$ Since $|(\alpha-e_1)|=d-2=(d-1)-1,$ by induction $b_{(\alpha-e_1)}(\fil')\geq 0.$ Hence for $|\alpha|=d-1,$ $e_{\alpha}(\fil)\geq 0.$
\\$(2)$ We prove the result using induction on $d.$ For $d=2$ the result follows from the Difference Formula (Theorem \ref{r3}) for $\n=\underline 0$ and Proposition \ref{zero}. The rest is same as for the case $|\alpha|=d-1.$
\eepf
\et
As a consequence of this we get the following results which is proved by Marley \cite[Proposition 3.19, Proposition 3.23]{marleythesis}.
The next one is a generalisation of a result due to M. Narita \cite[Theorem 1]{narita}. Here we give different proof.
\bp\label{v2}
Let $(R,\mm)$ be a $d$-dimensional $(d\geq 2)$ \CM local ring, $I$ an $\mm$-primary ideal and $\fil=\lbrace I^n\rbrace_{n\in{\ZZ}}$ an admissible $I$-filtration. Then $e_2{(\fil)}\geq 0.$ 
\ep
\bpf
Comparing the expressions of coefficients of Hilbert polynomials for $s=1$ and $s\geq 2,$ by Theorem \ref{all}, we get the required result.
\eepf

It is natural to ask whether $e_i(\fil)$ are nonnegative for $i \geq 3$ in a Cohen-Macaulay local ring. Narita \cite[Theorem 2]{narita} and Marley \cite[Example 2]{marley} gave an example of an ideal in a \CM local ring with $e_3(I) < 0 .$ 
\begin{example}\cite[Theorem 2]{narita}
Let $\Delta$ be a formal power series $k[[X_1,X_2,X_3,X_4]]$ over a field $k$ and $Q=\Delta/{\Delta}X_4^3.$ Then $Q$ is a \CM local ring of dimension $3.$ Let $x_1,x_2,x_3,x_4$ be the images of $X_1,X_2,X_3,X_4$ in $Q$ and $I=Qx_1+Qx_2^2+Qx_3^2+Qx_2x_4+Qx_3x_4.$ Then $$e_3(I)=-\lm_{Q'}\lf\frac{((IQ')^2: (x_2Q')^2)}{IQ'}\rg=-\lm_{Q'}\lf\frac{IQ'+(x_4Q')^2}{IQ'}\rg< 0\mbox{ where }Q'=Q/(x_1).$$
\end{example}
\begin{example} \cite[Example 2]{marley}
  Let  $I=(X^3,Y^3,Z^3,X^2Y,XY^2,YZ^2,XYZ)$ in the regular local ring  $R=k[X,Y,Z]_{(X,Y,Z)}.$ Then for all $n \geq 1,$ 
  $$P_I(n)=27{n+2\choose 3}-18{n+1\choose 2}+4n+1.$$
  Hence $e_3(I)=-1<0.$
 \end{example}

However, for $\fil=\{\ov{I^n}\}_{n \in \ZZ}$, Itoh proved that ${e}_3(\fil)$ is nonnegative in an analytically unramified Cohen-Macaulay local ring \cite[Theorem 3]{i}. In order to prove this, he used an analogue of Theorem \ref{GS}
(see \cite[p. 114]{i}). In \cite[Corollary 3.9]{huckaba-huneke}, authors gave an alternative proof of this result. We prove this result using the GSF. For this purpose, we recall  some results of Itoh about vanishing of graded components of local cohomology modules. See also \cite[Theorem 1.2]{hu}.

\begin{theorem}
{\rm \cite[Theorem 2]{i88}} {\rm \cite[Proposition 13]{i}}\label{vanishing of lc} Let $(R, \mathfrak{m})$ be an analytically unramified  Cohen-Macaulay local ring of
 dimension $d \geq 2.$ Let $\mathcal M=(t^{-1}, \R(\fil)_+)$ be the maximal homogeneous ideal of $\R'(\fil).$ Then the  following statements hold true for the filtration $\fil=\{\ov{I^n}\}_{n \in \ZZ}:$  \\
{\rm (1) } $H^0_{\mathcal M}(\mathcal{R}'(\fil))=H^1_{\mathcal M}(\mathcal{R}'(\fil))=0;$ \\
{\rm  (2) } $H^2_{\mathcal{M}}(\mathcal{R}'(\fil)_j=0$ for $j \leq 0;$ \\
{\rm (3)} $H^i_{\mathcal{M}}(\mathcal{R}'(\fil))=H^i_{\R(\fil)_+}(\mathcal{R}'(\fil))$ for $i=0,1,\ldots,d-1.$
\end{theorem}

\begin{theorem} \cite[Theorem 3]{i}
 Let $(R,\mm)$ be an analytically unramified Cohen-Macaulay local 
ring of dimension $d \geq 3$ and $I$ be an $\mm$-primary ideal in $R.$ Then $\ov{e}_3(I) \geq 0.$
\end{theorem}
\begin{proof} For $\fil=\{\ov{I^n}\}_{n \in \ZZ},$ we set $\ov{\R'}(I):=\R'(\fil).$ 
We use induction on $d.$ Let $d=3.$ Then, by the Difference Formula (Theorem \ref{r3}) for Rees algebras and  Theorem \ref{vanishing of lc}, we have 
\begin{eqnarray*}
\ov e_3(I)=h^3_{\ov{\mathcal R^{'}}(I)_+}\ov{\mathcal{ R}'}(I)_0\geq 0. 
\end{eqnarray*}
Let $d >3.$ We may assume that the residue field of $R$ is infinite. Let $J \subseteq I$ be a reduction of $I.$ Since $\ov{I^n}=\ov{J^n}$ for all $n,$ $\ov{e}_i(I)=\ov{e}_i(J)$ for all $i=1,\ldots, d.$ Therefore it suffices to show that $\ov{e}_3(J) \geq 0.$ By \cite[Theorem 1 and Corollary 8]{i}, there exists a system of generators $x_1,\ldots,x_d $ of $J$ such that, if we put $T=(T_1,\ldots,T_d), R(T)=R[T]_{\mm[T]}$ and $ C=R(T)/(\sum_{i=1}^d x_i T_i)$, then $C$ is an analytically unramified Cohen-Macaulay local ring of dimension $d-1$ and $\ov{e}_3(J)=\ov{e}_3(JC).$ Hence, by using induction hypothesis the result follows. 
\end{proof}

\noindent
Itoh \cite[p. 116]{i} proposed the following conjecture on the vanishing of $\ov{e}_3(I)$ which is still open.

\begin{conjecture}
\label{itoh's conjecture}
{\rm (Itoh's Conjecture):} 
 Let $(R,\mm)$ be an analytically unramified Gorenstein local ring of dimension $d\geq 3$. 
Then $\ov e_3(I)=0$ if and only if $\ov{I^{n+2}}=I^n \ov{I^2}$ for every $n \geq 0.$
\end{conjecture}

 Itoh proved the  ``if'' part of the Conjecture \ref{itoh's conjecture} in \cite[Proposition 10]{i88}. He also proved 
the ``only if'' part of the Conjecture \ref{itoh's conjecture} for $\ov{I}=\mm$ \cite[Theorem 3(2)]{i}. By \cite[Corollary 8 and Proposition 17]{i}, it suffices to prove the Conjecture \ref{itoh's conjecture} for $d=3.$  Let $d=3$ and $\ov{e}_3(I)=0$ for an $\mm$-primary ideal in a Cohen-Macaulay ring $R.$ By \cite[Proposition 3]{i88} and \cite[Corollary 16 and (4.1)]{i}, $\ov{I^{n+2}}=I^n \ov{I^2}$ for every $n \geq 0$ if and only if $\ov{\R^{'}}(I)$ is Cohen-Macaulay.

It is not known whether the Itoh's conjecture is true for $\ov{I}=\mm$ in a Cohen-Macaulay local ring $R$ (which need not be Gorenstein). Recently, in \cite[Theorem 3.6]{cpr}, the  authors proved that the Conjecture \ref{itoh's conjecture} holds true for $\ov{I}=\mm$ in a Cohen-Macaulay local ring of type at most two.  T. T. Phuong \cite{P}, showed that if $R$ is an analytically unramified \CM local ring of dimension $d\geq 2$ then the equality $\ov{e}_1(I)=\ov{e}_0(I)-\lm(R/\ov I)+1$ leads to the vanishing of $\ov{e}_3(I)=0.$ In \cite{km}, authors generalised the result of \cite{cpr}. They also obtained following result for an arbitrary $\mm$-primary ideal $I$ in an analytically unramified Cohen-Macaulay local ring of dimension $3.$ 

\begin{theorem} \cite[Theorem 1.1]{km}
\label{theorem:itoh}
Let $(R,\mm)$ be an analytically unramified Cohen-Macaulay local ring of dimension $3.$ 
Let $\mathcal M=(t^{-1},\ov{\R^{'}}_+)$ and $ \ov{\R^{'}}=\bigoplus\limits_{n \in \ZZ } \overline{I^n}t^n.$ 
Suppose that $\bar e_3(I) = 0.$
Then
\begin{enumerate}

\item \label{enum:thmitohHthreelocalvanishes}
$\displaystyle H^3_{\mathcal M}(\ov{\R^{'}}) = 0$, 



\item \label{enum:thmitohEquichar}
Suppose either that $R$ is equicharacteristic or that $\overline{I} = \mm,$ and that $I$ has a reduction generated by $x,y,z.$ If $\ov{\R^{'}}$ is not Cohen-Macaulay, then $\displaystyle \ov{e}_2(I)- \lm\left(\frac{\overline{I^2}}{(x,y,z)\overline{I}}\right) \geq 3$.
\end{enumerate}
\end{theorem}
As a consequence they generalised \cite[Theorem 3.6]{cpr}.
\begin{corollary} \cite[Corollary 1.2]{km}
\label{corollary:cpr}
Let $(R,\mm)$ be an analytically unramified Cohen-Macaulay local ring of dimension $3.$ 
\begin{enumerate}
\item \label{corollary:part1}
Suppose $\ov{e}_3(I)=0.$ Then there is an inclusion $H^3_{\ov{\R^{'}}_+}(\ov{\R^{'}})_{-1} \subseteq 
(0 :_{\homology^3_\mm(R)} \overline{I}).$

\item \label{corollary:part2} Suppose $\ov{e}_3(\mm)=0.$ Then $\bar e_2(\mm) \leq \type(R).$ 

\item \label{corollary:part3}
$\ov{\R^{'}}(\mm)$ is Cohen-Macaulay if
$\bar e_2(\mm) \leq \length_R(\overline{I^2}/\mm I) + 2$
for any ideal $I$ such that $\overline{I} = \mm,\mbox{ } \ov{e}_3(\mm)=0$ and $I$ has a minimal reduction. 
\end{enumerate}
\end{corollary}
\begin{proof}
\eqref{corollary:part1}: 
By Theorem \ref{theorem:itoh}, $H^3_{\mathcal M}(\ov{\R^{'}})=0.$ Hence, by \cite[Proposition 13(3)]{i}, we get an exact sequence
$$
0 \longrightarrow H^3_{\ov{\R^{'}}_+}(\ov{\R^{'}})_{-1} \longrightarrow H^3_{\mm}(R) \longrightarrow H^4_{\mathcal M}(\ov{\R^{'}})_{-1} \longrightarrow 0.
$$
Thus $H^3_{\ov{\R^{'}}_+}(\ov{\R^{'}})_{-1} \subseteq
{H^3_\mm(R)}.$
By the Difference Formula (Theorem \ref{r3}) and Theorem \ref{vanishing of lc}, we get 
\begin{equation}
 \label{equation:e3}
h^3_{\ov{\R^{'}}_+}(\ov{\R^{'}})_{0}=\bar{e}_3(I)=0. 
\end{equation}
Now consider the exact sequence 
 $$
 0 \longrightarrow \ov{\R^{'}}(1) \longrightarrow \ov{\R^{'}} \longrightarrow \ov{G}=\bigoplus\limits_{n\geq 0}\frac{\ov{I^n}}{\ov{I^{n+1}}} \longrightarrow 0 $$
which gives the long exact sequence 
$$ \cdots \longrightarrow H^i_{\ov{\R^{'}}_+}(\ov{\R^{'}})_{n+1} \longrightarrow H^i_{\ov{\R^{'}}_+}(\ov{\R^{'}})_{n}\longrightarrow H^i_{\ov{G}_+}(\ov{G})_{n} \longrightarrow \cdots. $$
Using \eqref{equation:e3}, we get an isomorphism 
$ H^3_{\ov{\R^{'}}_+}(\ov{\R^{'}})_{-1} \simeq H^3_{\ov{G}_+}(\ov{G})_{-1}.$
This implies that $ H^3_{\ov{\R^{'}}_+}(\ov{\R^{'}})_{-1}$ is an $R/\ov{I}$-module. Therefore 
$H^3_{\ov{\R^{'}}_+}(\ov{\R^{'}})_{-1} \subseteq 
(0 :_{\homology^3_\mm(R)} \overline{I}).$

\eqref{corollary:part2} 
Taking $I=\mm,$ by the Difference Formula (Theorem \ref{r3}) and Theorem \ref{vanishing of lc}, we get
$\bar{e}_2(I)=h^3_{\ov{\R^{'}}_+}(\ov{\R^{'}})_{-1}.$ Hence by \eqref{corollary:part1} we get 
the result.

\eqref{corollary:part3} 
Follows from \ref{theorem:itoh}\eqref{enum:thmitohEquichar}. 
 \end{proof}

The next result was first proved by Sally \cite[Proposition 5]{sa} for the filtration $\{\mm^n\}_{n\in\ZZ}$ and then by Johnston and Verma \cite[Proposition 3.3]{1jv} for the filtration $\{I^n\}_{n\in\ZZ}$ where $I$ is an $\mm$-primary ideal of $R.$ Here we prove the result for $\ZZ$-graded admissible filtrations.
\bp\label{co}
\label{formula for lc}
Let $(R,\mm)$ be a two-dimensional \CM local ring, $I$ be any $\mm$-primary ideal of $R$ and $\fil=\lbrace I_n\rbrace_{n\in{\ZZ}}$ an $I$-admissible filtration of ideals in $R.$ Then 
\ben{
\item $\displaystyle{\lm\lf H_{\R(\fil)_+}^2(\R(\fil)_0\rg=e_2(\fil)},$
\item $\displaystyle{\lm\lf H_{\R(\fil)_+}^2(\R(\fil))_1\rg=e_0(\fil)-e_1(\fil)+e_2(\fil)-\lm\lf\frac{R}{\breve {I_1}}\rg},$
\item $\displaystyle{\lm\lf H_{\R(\fil)_+}^2(\R(\fil))_{-1}\rg=e_1(\fil)+e_2(\fil)}.$}\een 
\ep
\bpf
We have \beqnn\label{b2}
\po(n)-\ho(n)=\sum\limits_{i\geq 0}(-1)^ih_{\R(\fil)_{+}}^i(\mathcal{R}'(\fil))_{n}\mbox{ for all }
 n\in\ZZ.\eeqnn
\\$(1)$ Putting  $n=0$ in  (\ref{b2}) and using Propositions \ref{result 2} and \ref{h1}, we get the required result.
\\$(2)$ Putting $n=1$ in  (\ref{b2}) and using Propositions \ref{result 2} and \ref{h1}, we get the required result.
\\$(3)$ Consider the short exact sequence of $\R(\fil)$-modules $$0\longrightarrow \R(\fil)_+\longrightarrow \R(\fil)\longrightarrow R\cong\R(\fil)/\R(\fil)_+\longrightarrow 0$$ which induces a long exact sequence of local cohomology modules whose $n$-th component is \beqn\cdots\longrightarrow H_{\R(\fil)_+}^i(\R(\fil)_+)_n\longrightarrow  H_{\R(\fil)_+}^i(\R(\fil))_n \longrightarrow  H_{\R(\fil)_+}^i(R)_n\longrightarrow  \cdots\mbox{ for all }i\geq 0.\eeqn Since $R$ is  $\R(\fil)_+$-torsion, $H_{\R(\fil)_+}^0(R)=R$ and $H_{\R(\fil)_+}^i(R)=0$ for all $i\geq 1.$ Hence $H_{\R(\fil)_+}^i(\R(\fil)_+)\cong  H_{\R(\fil)_+}^i(\R(\fil))$ for all $i\geq 2$ and we have the exact sequence 
\beqn 0\rightarrow H_{\R(\fil)_+}^0(\R(\fil)_+)_n\rightarrow  H_{\R(\fil)_+}^0(\R(\fil))_n \rightarrow  R\rightarrow\eeqn\beqnn\label{poo} H_{\R(\fil)_+}^1(\R(\fil)_+)_n\rightarrow H_{\R(\fil)_+}^1(\R(\fil))_n\rightarrow 0. \eeqnn 
\\The  short exact sequence of $\R(\fil)$-modules $$0\longrightarrow \R(\fil)_+(1)\longrightarrow \R(\fil)\longrightarrow G(\fil)\longrightarrow 0$$ induces the exact sequence 
\beqn 0\longrightarrow H_{\R(\fil)_+}^0(G(\fil))_{-1}\rightarrow  H_{\R(\fil)_+}^1(\R(\fil)_+)_0 \rightarrow H_{\R(\fil)_+}^1(\R(\fil))_{-1}\rightarrow H_{\R(\fil)_+}^1(G(\fil))_{-1}\rightarrow\eeqn\beqnn\label{too} H_{\R(\fil)_+}^2(\R(\fil))_0\rightarrow H_{\R(\fil)_+}^2(\R(\fil))_{-1}\rightarrow H_{\R(\fil)_+}^2(G(\fil))_{-1}\rightarrow 0 .\eeqnn Now $H_{\R(\fil)_+}^0(G(\fil))\subseteq G(\fil)$ are nonzero only in nonnegative  degrees. Thus $H_{\R(\fil)_+}^1(\R(\fil))_{-1}\cong R$ and $H_{\R(\fil)_+}^1(\R(\fil))_{0}=0$ by Proposition \ref{h222}. Therefore from the exact sequence (\ref{poo}), we get the exact sequence $$0\rightarrow R\rightarrow  H_{\R(\fil)_+}^1(\R(\fil)_+)_0 \rightarrow H_{\R(\fil)_+}^1(\R(\fil))_{0}=0.$$ Let $f$ denote the map from $H_{\R(\fil)_+}^1(\R(\fil))_{-1}$ to $H_{
\R(\fil)_+}^0(G(\fil))_{-1}$ in the exact sequence (\ref{too}). First we prove that $f$ is zero map. From the exact sequence (\ref{too}), we get the exact sequence $$0\longrightarrow R\overset{g}\longrightarrow R\overset{f}\longrightarrow H_{\R(\fil)_+}^1(G(\fil))_{-1}.$$ Since $R/g(R)$ is contained in $H_{\R(\fil)_+}^1(G(\fil))_{-1}$ and by Proposition \ref{vl}, $H_{\R(\fil)_+}^1(G(\fil))_{-1}$ is of finite length, we have $\lm_{R}\lf R/g(R)\rg$ is finite. Since $g(R)$ is principal ideal in $R,$ we get $R=g(R).$ Therefore $f$ is the zero map. Hence we get the exact sequence $$0\rightarrow H_{\R(\fil)_+}^1(G(\fil))_{-1}\rightarrow H_{\R(\fil)_+}^2(\R(\fil))_0\rightarrow H_{\R(\fil)_+}^2(\R(\fil))_{-1}\rightarrow H_{\R(\fil)_+}^2(G(\fil))_{-1} \rightarrow 0.$$ Therefore by  Theorem \ref{GS}, we get
\beqn
[\ho(0)-\ho(-1)]-[\po(0)-\po(-1)]&=& -\lm\lf H_{\R(\fil)_+}^1(G(\fil))_{-1}\rg+\lm\lf H_{\R(\fil)_+}^2(G(\fil))_{-1}\rg\\&=& -\lm\lf H_{\R(\fil)_+}^2(\R(\fil))_{0}\rg+\lm\lf H_{\R(\fil)_+}^2(\R(\fil))_{-1}\rg.
\eeqn
Thus by part $(1)$ of the Proposition, we get $$\lm\lf H_{\R(\fil)_+}^2(\R(\fil))_{-1}\rg=e_2(\fil)-e_2(\fil)+e_1(\fil)+e_2(\fil)=e_1(\fil)+e_2(\fil).$$

\eepf

\section{Huneke-Ooishi Theorem and a multi-graded version} 
In this section we give an application of the GSF to derive a result of Huneke \cite{huneke} and Ooishi \cite{ooishi} which states that if $(R,\mm)$ is a \CM local ring of dimension $d\geq 1$ and $I$ is an $\mm$-primary ideal then $e_0(I)-e_1(I)=\lambda(R/I)$ if and only if $r(I)\leq 1.$ A similar result for admissible filtrations  was proved in \cite[Theorem 4.3.6]{blancafort thesis} and \cite[Corollary 4.9]{huckaba-marley}. In \cite[Theorem 5.5]{msv}, authors gave a partial generalization of this result for an $\I$-admissible filtration.  First we prove few preliminary results needed.

\bl{\em(Sally machine)}\cite[Corollary 2.4]{sa1} \cite[Lemma 2.2]{huckaba-marley}\label{sm}
Let $(R,\mm)$ be a Noetherian local ring, $I_1$ an $\mm$-primary ideal in $R$ and $\fil=\lbrace I_n\rbrace_{n\in{\ZZ}}$ be an $I_1$-admissible filtration of ideals in $R.$ Let $x_1,\dots,x_r$ be a superficial sequence for $\fil.$ If $\grade G(\fil/(x_1,\dots,x_r))_+\geq 1$ then $\grade G(\fil)_+\geq r+1.$
\bpf
We use induction on $r.$ Let $r=1$ and $y\in I_t$ such that image of $y$ in $G(\fil/(x_1))_t$ is a nonzerodivisor. Then $(I_{n+tj}:y^j)\subseteq(I_n,x_1)$ for all $n,j.$ Since $x_1$ is a superficial element for $\fil,$ there exists integer $c\geq 0,$ such that $(I_{n+j}:x_1^j)\cap I_c=I_n$ for all $j\geq 1$ and $n\geq c.$ Consider an integer $p>c/t.$ For arbitrary $n$ and $j\geq 1,$ we prove that  $$y^p(I_{n+j}:x_1^j)\subseteq (I_{n+j+tp}:x_1^j)\cap I_c=I_{n+tp}.$$ Let $a\in(I_{n+j}:x_1^j).$ Then $ay^px_1^j\in I_{n+j+tp}.$ Since $pt>c,$ $ay^p\in (I_{n+j+tp}:x_1^j)\cap I_c=I_{n+tp}.$ Therefore $$(I_{n+j}:x_1^j)\subseteq (I_{n+tp}:y^p)\subseteq (I_n,x_1).$$ Thus $(I_{n+j}:x_1^j)=I_n+x_1(I_{n+j}:x_1^{j+1})$ for all $n$ and $j\geq 1.$ Iterating this formula $n$ times, we get $$(I_{n+j}:x_1^j)=I_n+x_1I_{n-1}+x_1^2I_{n-2}+\cdots+x_1^n(I_{n+j}:x_1^{j+n})=I_n.$$ Hence $x_1^*=x_1+I_2$ is a nonzerodivisor of $G(\fil).$ Since $G(\fil)/(x_1^*)\simeq G(\fil/(x_1)),$ $\grade G(\fil)_+\geq 2.$
\\Now assume $r\geq 2.$ Then by $r=1$ case, we have $\grade G(\fil/(x_1,\dots,x_{r-1}))_+\geq 2>1.$ By induction on $r,$ we have $\grade G(\fil)_+\geq r$ and since $x_1,\dots,x_r$ is a superficial sequence for $\fil,$ by Lemma \ref{hhh}, we obtain $x_1^*,\ldots,x_r^*$ is a regular sequence of $G(\fil).$ Since $G(\fil)/(x_1^*,\ldots,x_r^*)\simeq G(\fil/(x_1,\dots,x_r)),$ $\grade G(\fil)_+\geq r+1.$
\eepf 
\el
The next lemma is due to Marley \cite[Lemma 3.14]{marleythesis}.
\begin{lemma}\label{posn1}
Let $(R,\mm)$ be a \CM local ring of dimension $d\geq 1,$ $I$ an $\mm$-primary ideal and $\fil=\{I_n\}_{n\in \ZZ}$ be an $I$-admissible filtration of ideals in $R. $ Suppose $x\in I_1\setminus I_2$ such that $x^*=x+I_2$ is a nonzerodivisor in $G(\fil).$ Let $R^\prime=R/(x).$ Then $n(\fil)=n(\fil^\prime)-1$ where $\fil^\prime=\{I_n R^\prime\}_{n\in\ZZ}.$
\end{lemma}
\bpf
We use the notation `` $'$ " to denote the image in $R^\prime.$ For all $n,$ consider the following short exact sequence of $R$-modules
$$0\longrightarrow (I_n:x)/I_n\longrightarrow R/I_n\overset{.x}\longrightarrow R/I_n\longrightarrow R'/I_n'\longrightarrow 0.$$ Therefore $H_{\fil'}(n)=\lm( R'/I_n')=\lm((I_n:x)/I_n).$ Since $x^*$ is a nonzerodivisor in $G(\fil),$ we have $(I_{n+1}:x)=I_n$ for all $n.$ Hence $H_{\fil'}(n)=\lm(I_{n-1}/I_n)=\lm(R/I_n)-\lm(R/I_{n-1})=H_{\fil}(n)-H_{\fil}(n-1)$ for all $n$ which implies  $P_{\fil'}(n)=P_{\fil}(n)-P_{\fil}(n-1)$ for all $n.$ Thus $H_{\fil'}(n)=P_{\fil'}(n)$ for all $n\geq n(\fil)+2.$ Since \beqn P_{\fil'}(n(\fil)+1)-H_{\fil'}(n(\fil)+1)&=&[P_{\fil}(n(\fil)+1)-H_{\fil}(n(\fil)+1)]-[P_{\fil}(n(\fil))-H_{\fil}(n(\fil))]\\&=&-[P_{\fil}(n(\fil))-H_{\fil}(n(\fil))]\neq 0,\eeqn we get the required result.
\eepf
The next theorem is due to Blancafort \cite{blancafort thesis} which is a generalization of a result of Huneke \cite{huneke} and Ooish \cite{ooishi} proved independently. We make use of reduction number and postulation number of admissible filtration of ideals to simplify her proof.
\bt\cite[Theorem 4.3.6]{blancafort thesis}\label{HUOI}
Let $(R,\mm)$ be a \CM local ring with infinite residue field of dimension $d\geq 1,$ $I_1$ an $\mm$-primary ideal and $\fil=\lbrace I_n\rbrace_{n\in{\ZZ}}$ be an $I_1$-admissible filtration of ideals in $R.$ Then the following are equivalent:
\ben{
\item $e_0(\fil)-e_1(\fil)=\lm\lf{R}/{I_1}\rg,$
\item $r(\fil)\leq 1.$
}\een
In this case, $e_2(\fil)=\cdots=e_d(\fil)=0,$ $G(\fil)$ is Cohen-Macaulay, $ n(\fil) \leq 0,$ $r(\fil)$ is independent of the reduction chosen and  $\fil=\lbrace I_1^n\rbrace .$
\et
\bpf
$(1)\Rightarrow(2)$ We use induction on $d.$ Let $d=1.$  For all $n\in\ZZ,$ we have $$\po(n)-\ho(n)=-h_{\R(\fil)_+}^1(\R^\prime(\fil))_n.$$ By putting $n=1$ in this formula, we get  $e_0(\fil)-e_1(\fil)-\lm\lf{R}/{I_1}\rg=-h_{\R(\fil)_+}^1(\R^\prime(\fil))_1=0.$ Therefore by Lemma \ref{result 4}, for all $n\geq 1,$ $h_{\R(\fil)_+}^1\R^\prime(\fil)_n=0.$ Consider the short exact sequence of $\R(\fil)$-modules, 
\beqn
0\longrightarrow \mathcal{R}'(\fil)(1)\overset{t^{-1}}\longrightarrow \mathcal{R}'(\fil)\longrightarrow G(\fil)
\longrightarrow 0.\eeqn This induces a long exact sequence , 
$$
 0\longrightarrow [H_{{\R(\fil)}_{+}}^0( G(\fil))]_{n} \longrightarrow [H_{{\R(\fil)}_{+}}^1(\mathcal{R}'(\fil))]_{n+1} 
 \longrightarrow\cdots.$$ 
Thus for all $n\in\NN,$ $[H_{{\R(\fil)}_{+}}^0( G(\fil))]_{n}=0.$ Hence $G(\fil)$ is Cohen-Macaulay. Let $J=(x)$ be a minimal reduction of $\fil.$ 
 \wlg $x$ is superficial. 
For each $n,$ consider the following map  
$$\displaystyle\frac{I_{k+n}}{x^k I_{n}}\overset{\phi_k}\longrightarrow \frac{I_{k+n+1}}{x^{k+1} I_{n}}\mbox{ where }\phi_k(\ov z)=\ov {xz}.$$ For all large $k,$ $I_{k+n+1}=x I_{k+n}.$ Hence for all large $k,$ $\phi_k$ is surjective. Now suppose $\phi_k(\ov z)=0$ for some $\ov z\in {I_{k+n}}/{x^k I_{n}}.$ Then $x z\in x^{k+1} I_{n}.$ Therefore $xz=x^{k+1}a$ where $a\in I_n,$ hence $z\in x^kI_{n}.$ Thus for all large $k,$ $\phi_k$ is injective. Therefore by Proposition \ref{hlcr}, for all large $k,$ $$H_{\R(\fil)_+}^1(\R(\fil))_n\simeq \frac{I_{k+n}}{x^k I_{n}}.$$
 \\By Lemma \ref{result 4} and Proposition \ref{result 2}, $H_{\R(\fil)_+}^1(\R(\fil))_n=0$ for all $n\geq 1.$ Then for all large $k$ and $n\geq 1,$ $$I_{k+n}=x^k I_{n}.$$ Let $a\in I_{k+n-1}.$ Then $xa\in I_{k+n}\subseteq x^kI_n$ implies $a\in x^{k-1} I_{n}.$ Thus $I_{k+n-1}=x^{k-1} I_{n}.$ Using this procedure repeatedly we get $I_{n+1}=x I_{n}.$ Thus $r(\fil)\leq 1.$
 \\Let $d \geq 2$ and $x\in I_1$ be a superficial element for $\fil.$ Let $R^\prime ={R}/{(x)},$ $\fil^\prime=\{I_n R^\prime\}_{n\in\ZZ}$ and $G^{\prime}=G(\fil^\prime).$ Since $e_i(\fil)=e_i(\fil^\prime)$ for all $i<d,$ we have $$e_0(\fil^\prime)-e_1(\fil^\prime)=e_0(\fil)-e_1(\fil)=\lm\lf\frac{R}{I_1}\rg=\lm\lf\frac{R^\prime}{I_1R^\prime}\rg.$$ Hence by induction hypothesis, $G^\prime$ is Cohen-Macaulay. Therefore by Sally machine (Lemma \ref{sm}), $G(\fil)$ is Cohen-Macaulay. This implies that for any minimal reduction $J$ of $\fil,$ $r_J(\fil)=n(\fil)+d$ by Theorem \ref{relating pn and rn}. Thus $r_J(\fil)$ is independent of  the minimal reduction $J$ of $I.$ Let $J$ be a minimal reduction of $\fil$ generated by superficial sequence $x_1,\ldots,x_d.$ Let $\ov R=R/(x_1,\ldots,x_{d-1})$ and $\ov \fil=\{I_n\ov R\}_{n\in\ZZ}.$ Since $G(\fil)$ is \CM and $x_1,\ldots,x_d$ is superficial, using Theorem \ref{relating pn and rn} and Lemmas \ref{hhh}, \ref{redn1} and \ref{posn1}, for $d-1$ times, by induction hypothesis we get $$r(\fil)=n(\fil)+d=n(\ov\fil)+1=r(\ov\fil)\leq 1.$$
  $(2)\Rightarrow(1)$ Let $J$ be a minimal reduction of $\fil$ such that $r(\fil)=r_J(\fil)$ and $J$ is generated by superficial sequence $x_1,\ldots,x_d.$ Let $R^\prime=R/(x_1,\ldots,x_{d-1})$ and $\fil^\prime=\{I_nR^\prime\}_{n\in\ZZ}.$ Then $x_dI_nR^\prime=I_{n+1}R^\prime$ for all $n\geq 1.$ Since $x_d^\prime$ is nonzerodivisor, $(I_{n+1}R^\prime:x_d^\prime)=I_{n}R^\prime$ for all $n\geq 1.$ Therefore $(x_d^\prime)^*$ (the image of $x_d^\prime$ in $G(\fil^\prime)$) is nonzerodivisor in $G(\fil^\prime).$ Hence $G(\fil^\prime)$ is Cohen-Macaulay. Thus by Lemma \ref{sm}, $G(\fil)$ is Cohen-Macaulay. Therefore by Theorem \ref{relating pn and rn}, $n(\fil)=r(\fil)-d\leq 0.$
  Hence $\po(n)=\ho(n)$ for all $n > 0.$ By putting $n=1$ for $d=1$ case we obtain $e_0(\fil)-e_1(\fil)=\lm(R/I_1).$
\\Now we prove that if $r(\fil)\leq 1$ then $e_2(\fil)=\cdots=e_d(\fil)=0.$ \wlg assume $d\geq 2.$ The condition $r(\fil)\leq 1$ implies $G(\fil)$ is Cohen-Macaulay and $n(\fil)=r(\fil)-d<0.$ Let $d=2.$ Therefore $e_2(\fil)=\po(0)-\ho(0)=0.$ Now assume $d\geq 3$ and the result is true upto dimension $d-1.$ Let $J$ be minimal reduction of $\fil$ generated by superficial sequence $x_1,\ldots,x_d.$ Let $R^\prime=R/(x_1,\ldots,x_{d-1})$ and $\fil^\prime=\{I_nR^\prime\}_{n\in\ZZ}.$ Then $e_i(\fil)=e_i(\fil^\prime)=0$ for all $0\leq i<d.$ Since $G(\fil)$ is \CM and $n(\fil)=r(\fil)-d< 0,$ we get $(-1)^de_d(\fil)=\po(0)-\ho(0)=0.$ Therefore $e_0(\fil)-e_1(\fil)-\lm\lf{R}/{I_1}\rg=\po(1)-\ho(1)=0.$
 \\ Let $J$ be a minimal reduction of $\fil$ such that $r(\fil)=r_J(\fil)$ and $r(\fil)\leq 1.$ Then $I_{2}=JI_{1}\subseteq I_1^2\subseteq I_2.$ Suppose $I_r=I_1^r$ for all $1\leq r\leq n.$ Then $I_{n+1}=JI_{n}\subseteq I_1I_1^n\subseteq I_1^{n+1}\subseteq I_{n+1} .$ Thus $\fil$ is $\lbrace I_1^n\rbrace_{n\in{\ZZ}}.$
\eepf

\begin{theorem}\cite[Theorem 5.5]{msv}\label{huneke-ooishi dimension d}
Let $(R,\mm)$ be a Cohen-Macaulay local ring of dimension $d\geq 1$ and $\idl$ be $\mm$-primary ideals of $R.$ Let $\fil=\lbrace\fil(\n)\rbrace_{\n\in \ZZ^s}$ be an $\I$-admissible filtration of ideals in $R.$ Then for all $\ii,$\\
$(1)$ $e_{(d-1)e_i}(\mathcal F) \geq e_1(\mathcal F^{(i)}),$\\
 $(2)$ $e(I_i) - e_{(d-1)e_i}(\mathcal F) \leq \lm(R/\mathcal F{(e_i)}),$\\
 $(3)$ $e(I_i) - e_{(d-1)e_i}(\mathcal F) = \lm(R/\mathcal F{(e_i)})$
 if and only if 
$r(\mathcal F^{(i)}) \leq 1$ and $e_{(d-1)e_i}(\mathcal F)=e_1(\mathcal F^{(i)}).$
\end{theorem}
\bpf
$(1)$ 
We apply induction on $d$. Let $d=1$. Then by Theorem \ref{r3}, 
 \begin{eqnarray*}
 P_{\mathcal F}(re_i)-\lm(R/\mathcal F{(re_i)})=
 -\lm_{R}[H^1_{\mathcal{R}_{++}}(\mathcal{R}^\prime
 (\mathcal F))]_{(re_i)} \mbox{ for all } r \geq 0. 
\end{eqnarray*}
Since $\mathcal F^{(i)}$ is $I_i$-admissible, we have $e(\mathcal F^{(i)})=e(I_i).$ Hence using $P_{\mathcal F^{(i)}}(r)=e(I_i)r-e_1(\mathcal F^{(i)}),$ we get
\begin{eqnarray*}
 P_{\mathcal F^{(i)}}(r)-\lm(R/\mathcal F{(re_i)})+[e_1(\mathcal F^{(i)})-e_{\underline 0}(\mathcal F)]=
-\lm_{R}[H^1_{\mathcal{R}_{++}}(\mathcal{R}^\prime(\mathcal F))]_{(re_i)} \leq 0.
\end{eqnarray*} 
Taking $r \gg 0$, we get $e_{\underline 0}(\mathcal F)\geq e_1(\mathcal F^{(i)}).$ 
Let $d \geq 2.$ \wlg we may assume that the residue field of $R$ is infinite. By Lemma \ref{one}, there exists a nonzerodivisor $x_i \in I_i$ such that 
$$(x_i) \cap \mathcal F{(\n)}=x_i\mathcal F{(\n-e_i)}\mbox{ for }\n\in \NN^s\mbox{ where }n_i \gg 0.$$
Let $R^\prime=R/(x_i)$ and 
$\mathcal F^\prime=\{\mathcal F{(\n)}R^\prime\}$ and $\mathcal F^{\prime(i)}=\{\mathcal F{(ne_i)}R^\prime\}.$ For all $\n\in\NN^s$ such that $n_i \gg 0,$ consider the following exact sequence
$$
0\longrightarrow\frac{(\fil(\n):(x_i))}{\fil(\n-e_i)}\longrightarrow\frac{R}{\fil(\n-e_i)}\overset{.x_i}\longrightarrow\frac{R}{\fil(\n)}\longrightarrow\frac{R}{(x_i,\fil(\n))}\longrightarrow 0.$$ 
Since for all $\n\in\NN^s$ where $n_i \gg 0,$ $(\fil(\n):(x_i))=\fil(\n-e_i),$ we get 
$H_{\fil^\prime}(\n)=\ho(\n)-\ho(\n-e_i)$ and hence $\po(\n)-\po(\n-e_i)=P_{\fil^\prime}(\n).$ Therefore 
$e_{(d-2)e_i}(\mathcal F^\prime)=e_{(d-1)e_i}(\mathcal F)$ and $e_1(\mathcal F^{\prime(i)})=
e_1(\mathcal F^{(i)}).$ Therefore by induction, the result follows. 
\\$(2)$ Using part $(1),$ for all $\ii,$ we have
 $$e(I_i) - e_{(d-1)e_i}(\mathcal F) \leq e(I_i)-e_1(\mathcal F^{(i)})\leq \lm(R/\mathcal F{(e_i)})$$
 where the last inequality follows from Theorem \ref{n1}. 
\\$(3)$ Let $e(I_i) - e_{(d-1)e_i}(\mathcal F) = \lm(R/\mathcal F{(e_i)}) $. Then 
by part $(1),$
$$\lm(R/\mathcal F{(e_i)})=e(I_i) - e_{(d-1)e_i}(\mathcal F) \leq e(I_i) - e_{1}(\mathcal F^{(i)})
\leq \lm(R/\mathcal F{(e_i)}) ,$$
where the last inequality follows by Theorem \ref{n1}. 
Hence $e_{(d-1)e_i}(\mathcal F)=e_1(\mathcal F^{(i)})$ and 
$e(I_i) - e_{1}(\mathcal F^{(i)}) = \lm(R/\mathcal F{(e_i)})$. Therefore, by Theorem \ref{HUOI}, 
$r(\mathcal F^{(i)}) \leq 1.$ \\
Conversely, suppose $r(\mathcal F^{(i)}) \leq 1$ and $e_{(d-1)e_i}(\mathcal F)=e_1(\mathcal F^{(i)})$. 
Again, by Theorem \ref{HUOI}, $e(I_i)-e_1(\mathcal F^{(i)})=\lm(R/\mathcal F{(e_i)}).$ 
Hence $e(I_i) - e_{(d-1)e_i}(\mathcal F) = \lm(R/\mathcal F{(e_i)}).$
\eepf

\bt
\cite[Theorem 5.7]{msv}\label{e_2=0}
Let $(R,\mm)$ be a Cohen-Macaulay local ring of dimension two and $\idl$ be $\mm$-primary ideals of $R.$ Let $\fil=\lbrace\fil(\n)\rbrace_{\n\in\ZZ^s}$ be an $\I$-admissible filtration of ideals in $R.$ Then
$e_{\underline 0}(\fil)=0$ implies $e(I_i)-e_{e_i}(\fil)=\lm\lf\frac{R}{\breve\fil(e_i)}\rg$ for all $\ii.$ Suppose $\breve\fil$ is $\I$-admissible filtration, then the converse is also true.
\bpf
Let $e_{\underline 0}(\fil)=0.$ By Proposition \ref{zero}, 
$[H_{\R_{++}}^1(\mathcal{R}'(\fil))]_{\underline 0}
=0.$ Hence by Theorem \ref{r3}, $$\lm_{R}[H_{\R_{++}}^2(\mathcal{R}'(\fil))]_{\underline 0}=
e_{\underline 0}(\fil)=0.$$ By Lemma \ref{result 4}, $\lm_{R}[H_{\R_{++}}^2(\mathcal{R}'(\fil))]_{e_i}=0$ 
for all $\ii.$ Then using Theorem \ref{r3} and Proposition \ref{zero}, 
$\po(e_i)-\ho(e_i)=-\lm\lf\frac{\breve\fil(e_i)}{\fil(e_i)}\rg$ for all $\ii.$ 
Hence $e(I_i)-e_{e_i}(\fil)=\lm\lf\frac{R}{\breve\fil(e_i)}\rg$ for all $\ii.$ \\
Suppose $\breve\fil$ is $\I$-admissible filtration and $e(I_i)-e_{e_i}(\fil)=
\lm\lf\frac{R}{\breve\fil(e_i)}\rg$ for all $\ii.$ Then by \cite[Proposition 3.1]{msv} 
and Theorem \ref{two}, 
for all $\n\geq \underline{0}$ and $\ii,$ 
$$[H^0_{G_i(\breve{\mathcal F})_{++}}(G_i(\breve{\mathcal F}))]_{\n}=
\frac{\breve{\breve\fil}(\n+e_i)\cap{\breve\fil}(\n)}{\breve\fil(\n+e_i)}=0.$$ 
Since the Hilbert polynomial of $\breve\fil$ is same as the Hilbert polynomial of $\fil,$ by \cite[Theorem 5.3]{msv}, 
\beqnn \label{zzxx}
P_{\breve\fil}(\n)=H_{\breve\fil}(\n)  \mbox{ for all }\n \geq 0.
\eeqnn 
Thus taking $\n=\underline 0$ in the equation (\ref{zzxx}), we get $e_{\underline 0}(\fil)=e_{\underline 0}(\breve\fil)=0.$
\eepf
\et
As a consequence of the above theorem we get a theorem of Huneke \cite[Theorem 4.5]{huneke} for integral closure filtrations. We also obtain a result by Itoh \cite[Corollary 5]{itoh}  following from the above theorem.
\begin{corollary}\cite[Corollary 5]{itoh}\cite[Corollary 5.8]{msv}
Let $(R,\mm)$ be a Cohen-Macaulay local ring of dimension two and $I$ be $\mm$-primary ideal of $R.$ Let $Q$ be any minimal reduction of $I.$ Then the following are equivalent.
\ben{
\item[(1)] $e_1(I)-e_0(I)+\lm\lf\frac{R}{\breve I}\rg=0.$
\item[(2)] ${\breve I}^2=Q\breve I.$
\item[$(2')$] $\widebreve{I^2}=Q\breve I.$
\item[(3)] $\widebreve{I^{n+1}}=Q^n\breve I$ for all $n\geq 1.$
\item[(4)] $e_2(I)=0.$
}\een
\bpf
We prove $(4) \Rightarrow (3) \Rightarrow (2') \Rightarrow (2) \Rightarrow (1) \Rightarrow (4).$ \\
$(4)\Rightarrow(3):$ 
Let $\mathcal F=\lbrace \widebreve{I^n}\rbrace_{n\in\ZZ}.$ Since $e_2(\fil)=e_2(I)=0,$ by Theorem \ref{e_2=0} and Theorem \ref{HUOI}, the 
result follows.\\
$(3) \Rightarrow (2'):$ Put $n=1$ in $(3)$. \\ 
$(2')\Rightarrow(2)$ Consider the filtration 
$\fil=\lbrace I^n\rbrace_{n\in\ZZ}.$ Then by \cite[Proposition 3.2.3]{blancafort thesis}, 
for all $n\geq 0,$ $\widebreve{I^n}=\bigcup\limits_{k\geq 1}(I^{nk+n}:I^{nk}).$ 
It suffices to show that ${\breve I}^2\subseteq \widebreve{I^2}.$ Let $x,y\in {\breve I}.$ Then 
for some large $k,$ $xI^k\subseteq I^{k+1}$ and $yI^k\subseteq I^{k+1}.$ 
Hence $xyI^{2k}\subseteq I^{2k+2}.$ This implies that $\breve{I}^2\subseteq \widebreve{I^2}.$ \\
$(2)\Rightarrow(1):$ Follows from \cite[Theorem 2.1]{huneke}.\\ 
$(1)\Rightarrow(4):$ Let $\fil=\lbrace I^n\rbrace_{n\in\ZZ}.$ Since $\breve \fil$ 
is an $I$-admissible filtration, the result follows by Theorem \ref{e_2=0}. 
\eepf
\end{corollary}

\end{document}